\documentclass[a4paper,10pt]{amsart}

\usepackage[english]{babel}
\usepackage[latin1]{inputenc}
\usepackage{amsmath}
\usepackage{amsthm}
\usepackage{amssymb}
\usepackage{tikz}
\usepackage{nicefrac}
\usepackage{mathtools} 

\usepackage{imakeidx}
\usepackage{appendix}
\usepackage{tikz-cd}
\usepackage{verbatim}
\usepackage{enumitem}
\usepackage{multicol}

\usepackage{hyperref}
\usepackage{cleveref}

\usepackage{mathtools}

\usepackage{comment}

\newtheorem{theorem}{Theorem}[section]
\newtheorem{proposition}[theorem]{Proposition}
\newtheorem{corollary}[theorem]{Corollary}
\newtheorem{lemma}[theorem]{Lemma}

\newtheorem*{thm}{Theorem}

\theoremstyle{definition}

\newtheorem{definition}[theorem]{Definition}
\newtheorem{example}[theorem]{Example}

\newtheorem{remark}[theorem]{Remark}
\def\blue{\color[rgb]{0,0,1}}
\def\red{\color[rgb]{1,0,0}}
\def\green{\color[rgb]{0,1,0}}

\def\tor#1{\ensuremath{\operatorname{tor}#1}}
\def\torM#1{\ensuremath{\operatorname{tor}_{\mathcal{M}}{#1}}}
\def\dominio{\ensuremath{R}}

\DeclareMathOperator{\rk}{\operatorname{rk}}
\DeclareMathOperator{\cork}{\operatorname{cork}}
\DeclareMathOperator{\Ann}{\operatorname{Ann}}
\DeclareMathOperator{\Gr}{\operatorname{Gr}}

\DeclareMathOperator{\Hom}{\operatorname{Hom}}

\DeclareMathOperator{\link}{\operatorname{link}}

\title[Independencies and Tutte for $D$-matroids]{Set of independencies and Tutte polynomial\\ of matroids over a domain}
\author{Alessio Borz\`i and Ivan Martino}
\date{\today}

\begin{document}
	\begin{abstract}
		In this work, we study matroids over a domain and several classical combinatorial and algebraic invariants related.
		We define their Grothendieck-Tutte polynomial $T_{\mathcal{M}}(x,y)$, extending the definition given by Fink and Moci in \cite{fink2016matroids}, and we show that such polynomial has the classical deletion-contraction property.
		
		Moreover, we study the set of independencies for a realizable matroid over a domain, generalizing the definition of \emph{poset of torsions} $\Gr \mathcal{M}$ given by the second author in \cite{martino2017face}.
		This is a union of identical simplicial posets as for (quasi-)arithmetic matroids. The new notions harmonize naturally through the face module $N_\mathcal{M}$ of the matroid over a domain. 
		
		Whenever $\Gr \mathcal{M}$ is a finite poset, the Hilbert series $N_\mathcal{M}(t)$ of its face module is a specialization of the Tutte polynomial $T_{\mathcal{M}}(x,y)$.
		Further, for arrangements of codimension-one abelian subvarities of an elliptic curve admitting complex multiplication, we extend certain results of Bibby \cite{MR3487239} and we provide an algebraic interpretation of the elliptic Tutte polynomial.
	\end{abstract}
	\maketitle

	Let $N$ be a $n\times d$ matrix with coefficients in a domain $\dominio$,
	\begin{equation}\label{eq:matrix-intro}
	N = \begin{bmatrix}
	v_{1,1}     & v_{1,2} & \dots & v_{1,n} \\
	v_{2,1}     & v_{2,2} & \dots & v_{2,n} \\
	\vdots      & \vdots  & \ddots & \vdots \\
	v_{d,1}     & v_{d,2} & \dots & v_{d,n}
	\end{bmatrix}.
	\end{equation}
	We consider each column $v_i=(v_{1,i}, \dots, v_{d,i})^t$ to be an element in the $\dominio$-module $R^d$. 
	%
	When $\dominio$ is a field, the collection of linear independent subsets of $\{v_1, \dots, v_n\}$ is a realizable matroid. In other words, we consider all subsets $A$ of $[n]=\{1,\dots, n\}$ such that $\operatorname{span}(v_i:i\in A)$ has precisely dimension $|A|$. The collection of those sets is a simplicial complex, called \emph{independent complex}, see for instance the first chapter of \cite{oxley2006matroid}.
	
	A similar collection of independencies was studied for matrix with integer coefficients. The character group associated to each column is a disjoint union of algebraic subtori in $(\mathbb{C}^*)^d$ \cite{MR2183118}. This collection of tori is a toric arrangement that realized a so called arithmetic matroid \cite{MR2989987}.
	In this instance, a substitute for the independent complex is provided in \cite{martino2017face} by the \emph{poset of torsions} $\Gr \mathcal{M}$ associated to a matroid over $\mathbb{Z}$. 
	The poset of torsions for a represented Z-matroid is also a quotient of the independence poset of the semimatroid associated to an infinite periodic arrangement, see Section 7 of \cite{Delucchi-Dali-STANLEYREISNERRINGS}.
	\noindent
	The second author \cite{martino2017face} has shown that for a matroid over $\mathbb{Z}$ of rank $r$ $\Gr \mathcal{M}$ is simplicial (not anymore a simplicial complex), it has a face module $N_\mathcal{M}$ and the natural specialization of the arithmetic Tutte polynomial $T^{a}(x,y)$ of $\mathcal{M}$ is the Hilbert series $N_\mathcal{M}(t)$ of such face module,
	\[ 
	N_\mathcal{M}(t) = \frac{t^r}{(1-t)^r}T^{\operatorname{a}}(\nicefrac{1}{t},1). 
	\]

	In this manuscript we focus on matroids over a domain and we define their Grothendieck-Tutte polynomial, substantially extending the defintions in Section 7 of \cite{fink2016matroids}. Moreover, in the realizable case, we provide their poset of torsions, their Grothendieck $f$-vector, and their face module $N_\mathcal{M}$.
	We study the relation among these new algebro-combinatorial objects, but, before going any further with the presentation of the results, it is worth to explain the motivations driving us.
	
	\subsubsection*{Generalizations of matroids.}
	Matroids are at the crossroads of Algebra, Combinatorics, Geometry, and Topology. Since they were introduced by Whitney \cite{whitney1935} in 1935, new variations have appeared in literature encoding different type of independence. We have already presented arithmetic matroids \cite{moci2012tutte, MR2989987} that are closely related to valuated matroids \cite{MR1164708, MR1433646,MR2844079, MR3944531}. It is also important to mention the effort on the study of oriented matroids
	\cite{MR0485461, MR970389} and complex matroids \cite{MR1226982, MR3000567}.
	
	\noindent
	Of crucial importance for this work is the foundational article of Fink and Moci \cite{fink2016matroids} where they define and develop the theory of matroids (of modules) over a ring $R$. 
	A matroid over $R$ is not defined as a matroid enriched with extra data, but with a single axiom, naturally generalizing the classical matroid axioms.
	%
	%
	There are further generalizations to matroids over hyperings \cite{MR3702974, MR3771848, MR3725876, BakerBowlwerMatroidsHyper, MR3861778}, strongly related to tropical geometry, see also for instance 
	\cite{Frenkphdthesis, MR3783573}.

	\subsubsection*{Matroids over a domain $\dominio$.}
	In this paper we focus on matroids over a domain $\dominio$, that is the choice (up to isomphism) of a finitely generated $\dominio$-module $\mathcal{M}(A)$ for every subset $A$ of a ground set $[n]$ satisfying a technical but natural axiom, see Defintion \ref{definition matroid over ring}.
	
	\noindent
	A realizable matroid over $\dominio$ is a choice of elements $v_1, \dots, v_n$ in the finitely generated $R$-module $\mathcal{M}(\emptyset)$ such that $\mathcal{M}(A) \simeq \mathcal{M}(\emptyset) / (v_i : i \in A)$. Each realization is in bijection with the matricies (up to equivalence) of the form \eqref{eq:matrix-intro}. 
	%
	
	\noindent
	Matroids over a field (when the domain $\dominio$ is a field) are classical matroids \cite[Proposition 2.6]{fink2016matroids}: the choice of the field is not relevant, but it is relevant that $\mathbf{k}$-modules are $\mathbf{k}$-vector spaces.
	The classical matroid is so described by the rank function.
	
	\noindent
	Realizable matroids over $\mathbb{Z}$ coincide with arithmetic matroids \cite{moci2012tutte} and essentially encode the Combinatorics behind toric arrangements. 
	Realizable arthimetic matroids are characterized by Pagaria in \cite{Pagaria-Orientable-arithmetic-matroids}. 
	
	\noindent
	Fink and Moci have also characterized matroids over a Discrete Valuation Ring \cite{fink2016matroids} and more generically over a valuation ring \cite{MR3883211}.
	In this manuscript we initiate the natural next step of the study and we focus on matroid over a domain.

	\subsubsection*{The Grothendieck-Tutte polynomial}
	The Tutte polynomial of a classical matroid $M$ of rank $r$ is a polynomial with integer coefficients defined as 
	\begin{equation*}
	T_{M}(x,y)=\sum_{A\subseteq [n]} (x-1)^{r-\rk(A)}(y-1)^{|A|-\rk(A)},
	\end{equation*}
	where $\rk(A)$ is the rank of $A$.
	Moci \cite{moci2012tutte} provided its arithmetic version by correcting every term of the sum with the \emph{multiplicity} $m(A)$ of the set:
	\begin{equation*}
	T^{\operatorname{a}}_\mathcal{M}(x,y)=\sum_{A\subseteq [n]} m(A) (x-1)^{r-\rk(A)}(y-1)^{|A|-\rk(A)}.
	\end{equation*}
	In the realizable case, such multiplicity coincides algebraically with the cardinality of a certain torsion group, and geometrically with the number of connected components of a specific intersection in the toric arrangement described above {\cite{moci2012tutte, MR2989987, MR3240933}} realizing the (quasi-)arithmetic matroid $\mathcal{M}$. 
	
	\noindent
	For elliptic arrangements $\mathcal{E}$ \cite{MR3203648, PagariaElliptic, MR3487239}, by setting the multiplicity with the same geometric interpretation, but with no algebraic counterpart, Bibby has defined the elliptic Tutte polynomial $T^e_\mathcal{E}(x,y)$;
	this is provided in equation \eqref{eq:tutte-elliptic} and treated carefully at the end of Section 4 of \cite{MR3487239}. 
	As an application of our results we are going to provide an algebraic interpretation of the multiplicity $m(S)$ for elliptic arrangements of elliptic curve admitting complex muliplication. 
	Moreover, Pagaria \cite{PagariaElliptic} has recently shown that the Poincar\'e polynomial of the complement of an elliptic arrangements is not a specialization of the elliptic Tutte polynomial $T^e(x,y)$.
	It is worth to recall that Tran and Yoshinaga \cite{MR3916005}, also together with Ye \cite{YeTanMasahiko} have generalized the notion of Tutte polynomial to abelian Lie group arrangements.
	
	In Section 7 of \cite{fink2016matroids}, Fink and Moci described the Tutte-Grothendieck polynomial as an element of the Tutte-Grothendieck ring of an $R$-matroid $\operatorname{K}(R-\operatorname{Mat})$.
	Whenever $R$ is a Dedekind Domain, the Tutte-Grothendieck ring injects into a Grothendieck style ring and its elements resemble the "Tutte polynomial" for specific valuations. (We are going to be more specific in Section \ref{sec:preliminary-Tutte} of the Preliminaries.)
	One of our main results is to provide an \emph{explicit} form for the Tutte polynomial for matroid over a domain (not just Dedekind). This formulation does not involve using the dual of an $R$-matroid.

	\subsubsection*{The face poset and face ring of a matroid}
	One of the cryptomorphism for classical matroids is that a matroid $M$ is a simplicial complex $\mathcal{I}$ on a {\em ground set} $[n]$ satisfying the following property:
	\begin{equation*}
	A, B \in \mathcal{I}, |A|>|B| \Rightarrow \exists a\in A\setminus B: B\cup \{a\}\in \mathcal{I}.
	\end{equation*}
	%
	The simplicial complex $\mathcal{I}$ is called independent complex. The set $\mathcal{I}$ ordered by inclusion is the face poset of the matroid $M$, and the Stanley-Reisner ring $A_M$ of $M$ is the face ring of this poset. 
	%
	The Hilbert series of $A_M$ and the Tutte polynomial of $M$ are related by the following result:
	
	\begin{thm}[Appendix of Bj\"orner in \cite{DC-P-Box}]
		Let $M$ be a matroid of rank $r$ with ground set $[n]$ and call $M^*$ its dual matroid. Then:
		\[
		A_M(t)=\frac{t^{r}}{(1-t)^{r}} T_{M^*}(1, \nicefrac{1}{t}).
		\]
		where $A_M(t)$ is the Hilbert series of $A_M$.
	\end{thm}
	
	Recently the second author has generalized this result for realizable arithmetic matroids (Theorem A of \cite{martino2017face}) and, in this work, we are going to push even further by showing that the theorem holds also for every realizable matroids over a domain having finite poset of torsion.
	
	\section*{Presentation of our results}
	We briefly set some notations: if $\mathcal{M}$ is a matroid over a domain $\dominio$, then $\torM(A)$ is the the torsion submodule of $\mathcal{M}(A)$ and the rank (associated to the essential part of the generic matroid) is denoted by $\rk_{\mathcal{M}}(A)$.
	Moreover, $L_0(R\text{-mod})$ is the commutative ring generated by the isomorphism classes of finitely generated $R$-modules $[N]$ under the relation $[N\oplus N']=[N][N']$, which also defines the product operation in this ring.
	\setcounter{section}{2}
	\begin{definition}
		We define the Tutte polynomial of the $R$-matroid $\mathcal{M}$ as the following polynomial with coefficients in $L_0(R\text{-mod})$
		\[ 
		T_{\mathcal{M}}(x,y) = \sum_{A \subseteq [n]} [\tor(A)^\vee] (x-1)^{r-\rk(A)}(y-1)^{|A|-\rk(A)} \in L_0(R\text{-mod})[x, y]. 
		\]
	\end{definition}
	
	This polynomial fulfills all the expected properties:
	\setcounter{theorem}{7}
	\begin{theorem}[Deletion-Contraction property]
		Let $R$ be a domain and $\mathcal{M}$ be an $R$-matroid of rank $r$ on the ground set $[n]$. If $\mathcal{M}(\emptyset)$ is torsion free and $\mathcal{M}([n])=0$, then
		\[ 
		T_{\mathcal{M}}(x,y) = 
		\begin{cases}
		y T_{\mathcal{M} \setminus i}(x,y) & \text{if $i$ is a loop,} \\
		x T_{\mathcal{M} / i}(x,y) & \text{if $i$ is a coloop,} \\
		T_{\mathcal{M} \setminus i}(x,y) + T_{\mathcal{M} / i}(x,y) & \text{otherwise.}
		\end{cases} 
		\]
	\end{theorem}
	
	
	
	Finally, as in the classical case and in the arithmetic case, the Tutte polynomial does save a lot of combinatorial information. For instance, we can read back the $f$-vector of the matroid $\mathcal{M}$.
	
	\setcounter{section}{3}
	\begin{definition}
		Let $\mathcal{M}$ be a matroid over a domain $R$.
		Recall that $\Delta \mathcal{M}$ is the independent complex of the generic matroid of $\mathcal{M}$.
		We define 
		\[
		f_{i-1}=\sum_{\substack{A\in \Delta \mathcal{M}, \\ |A|=i}} [\tor (A)^{\vee}]\in L_0(R\text{-mod})
		\]
		The Grothendieck $f$-vector of $\mathcal{M}$ is the vector  $(f_{-1}, f_{0}, \dots, f_{r-1})$ in $L_0(R\text{-mod})^{r+1}$. 
	\end{definition}
	
	Whenever the poset of torsion is finite, we get back the classical notion of the $f$-vector, by evaluating the isomorphic classes $[\tor(A)^{\vee}]$ by their cardiality, see for instance \cite{MR3336842, MR2538614}. We are going to work a few interesting examples of these cases in Section \ref{Sec:ring-of-integers}.
	
	\noindent
	Moreover, even if the poset of torsions in not finite, the Tutte polynomial and the Grothendieck $f$-vector are related through
	\begin{equation*}
	\frac{t^r}{(1-t)^r} T_{\mathcal{M}}\left(\nicefrac{1}{t},1\right) = \sum_{i=0}^r f_{i-1} \frac{t^i}{(1-t)^i};
	\end{equation*}
	see Theorem \ref{thm:tutte-f-vector}.
	
	To present properly the poset of torsions of a matroid over a domain, we need to establish a few notations.
	From an $R$-matroid $\mathcal{M}$, we can associate the classical matroid $\mathcal{M}_E \otimes Q(R)$, where $Q(R)$ is the field of fractons of the domain $R$.
	We denote by $\Delta \mathcal{M}$, the simplicial complex of independent sets of $\mathcal{M}_E \otimes Q(R)$. 
	%
	%
	Let $A$ be a subset of $[n]$ and $b \in [n] \setminus A$ such that $A \cup \{b\} \in \Delta \mathcal{M}$. Then, of course, $A \in \Delta \mathcal{M}$ and $\psi(b) \in \mathcal{M}(\emptyset)$ is not a torsion element. 
	By Definition \ref{definition matroid over ring} there is the quotient map
	\[ 
	\overline{\pi_{A,b}}:\mathcal{M}(A) \rightarrow \mathcal{M}(A \cup \{b\}) \simeq \mathcal{M}(A) / (\psi(b)). 
	\]
	
	\noindent
	If we restrict $\overline{\pi_{A,b}}$ to $\torM{A}$, from Lemma \ref{injection of torsion}, we obtain an injective map denoted by
	\[ 
	\pi_{A,b}: \torM{A} \rightarrow \torM{A \cup \{b\}}. 
	\]
	We now consider the dual map, passing to the 
	controvariant functor $\Hom(-,\nicefrac{Q(R)}{R})$; we call $\tor (A)^{\vee} = \Hom(\tor (A),\nicefrac{Q(R)}{R})$. Thus we obtain the surjective map
	\[ 
	\pi_{A,b}^{\vee}:\tor (A \cup \{b\})^{\vee} \twoheadrightarrow \tor (A)^{\vee}. 
	\]
	
	
	\setcounter{section}{4}
	\setcounter{theorem}{1}
	\begin{definition}
		Let $\mathcal{M}$ be a realizable matroid over $R$, then
		\[ \Gr \mathcal{M} = \{ (A,l) : A \in \Delta \mathcal{M}, \, l \in \tor (A)^{\vee} \}, \]
		is the \emph{set of torsions} of $\mathcal{M}$. We define an order on $\Gr \mathcal{M}$ by providing the covering relations. If $(A \cup \{b\},h),(A,l) \in \Gr \mathcal{M}$, then we set
		\[ (A,l) \triangleleft (A \cup \{b\},h) \stackrel{\text{def}}{\Longleftrightarrow} \pi_{A,b}^{\vee}(h)=l. \]
		
	\end{definition}
	
	When the domain $R$ is a field, then $\mathcal{M}$ is a classical matroid and the poset of torsion coincides with the independent complex. 
	When $R$ is the integer ring, the poset of torsions was introduced by the second author in \cite{martino2017face}.
	Similarly as the arithmetic case, this new poset is simplicial:
	
	\setcounter{theorem}{7}
	\begin{theorem}
		If $\mathcal{M}$ is a realizable matroid over $R$, then $\Gr \mathcal{M}$ is a disjoint union of simplicial posets isomorphic to $\link (\emptyset,e)$.
	\end{theorem}
	
	This poset has always finite rank, but it may not be finite, see Example \ref{example:simplicial poset infinite}.
	If it is finite, the $f$-vector of the matroid does coincide with the $f$-vector of the poset, see Section \ref{sec-f-vector-ref-intro}.

	Classically, the keystone in between the Grothendieck-Tutte and the independent complex is the face ring \cite{MR782306, stanley1991f, stanley2007combinatorics, MR2110098}. Similarly, for matroids over a domain, we provide a new algebraic object, the face module $N_\mathcal{M}$.
	
	\noindent
	Let $\mathcal{M}$ be a realized matroid over a domain $R$, with no torsion in $\mathcal{M}(\emptyset)$.
	From Theorem \ref{poset of torsions simplicial poset}, $\Gr \mathcal{M}$ is a simplicial poset and we define $A_{\mathcal{M}}$ as its face ring following Stanley \cite{stanley1991f}.
	
	The \emph{face ideal} of $\Gr \mathcal{M}$ is the ideal of the polynomial ring $\mathbb{K}[x_a: a \in \Gr \mathcal{M}]$ defined as
	\[ 
	I_\mathcal{M} = \left( x_{\hat{0}} - 1, \: x_ax_b - x_{a \wedge b} \left( \sum_{c \in M(a,b)} x_c \right)  : a,b \in \Gr \mathcal{M} \right),
	\]
	where $M(a,b)$ denote the set of minimal upper bounds of $\{a,b\}$ and we require $M(a,b)$ to be a finite set, $x_{a \wedge b} = 0$ if $a \wedge b$ does not exists, and $\sum_{c \in M(a,b)}x_c = 0$ if $M(a,b) = \emptyset$.
	We set the degrees of $x_a$ following the rank of $a$: $\deg(x_a) = \rk(a)$ for all $a$ in $\Gr \mathcal{M}$, see Definition \ref{def:Face-ring}.
	
	
	
	
	%
	
	The \emph{face ring} of $\mathcal{M}$ is the quotient
	\[ 
	A_\mathcal{M} = \frac{\mathbb{K}[x_a:a \in P]}{I_\mathcal{M}}. 
	\]
	Whenever $\mathcal{M}(\emptyset)$ has torsions, 
	then we should define a \emph{face module}, $N_{\mathcal{M}}$.
	
	\setcounter{theorem}{10}
	\begin{definition}
		Let $L$ be the link of $(\emptyset, e)$ in $\Gr \mathcal{M}$ and denote by $A_L$ the face ring of $L$.
		The face module $N_{\mathcal{M}}$ of a matroid over a domain $R$ is the $A_L$-module,
		\[
		N_{\mathcal{M}}=A_L^{|tor(\emptyset)|}.
		\]
	\end{definition}
	
	Mimic the more intricate combinatorics of matroids over a domain, the face ring (or the face module) may be not Noetherian, see Example \ref{example:simplicial poset infinite}.
	
	\section*{Applications}
	There are two interesting application of our results:
	
	\subsubsection*{Hilbert series of the face module.}
	
	Let $R$ be a ring of integers of a number field. In Section \ref{Sec:ring-of-integers}, we define a homomorphism of rings $\varphi:L_0(R\text{-mod}) \rightarrow \mathbb{Z}$, by sending the class of every projective module to $1$ and the class of every torsion module to its cardinality.
	The homomorphism $\varphi$ induces the homomorphism of polynomial rings
	\[ 
	\tilde{\varphi}:L_0(R\text{-mod})[x,y] \rightarrow \mathbb{Z}[x,y]. 
	\]
	and, thus, we consider the image under $\tilde{\varphi}$ of the Tutte polynomial
	\begin{equation*}
	\tilde{T}_{\mathcal{M}}(x,y) = \tilde{\varphi} \big( T_{\mathcal{M}}(x,y) \big) = \sum_{A \subseteq [n]} |\tor(A)^\vee| (x-1)^{r-\rk(A)}(y-1)^{|A|-\rk(A)}.
	\end{equation*}
	
	\setcounter{section}{5}
	\setcounter{theorem}{3}
	\begin{theorem}
		If $\mathcal{M}$ is a realizable $R$-matroid of rank $r$ with finite poset of torsions and let $N_\mathcal{M}(t)$ be the Hilbert series of $N_\mathcal{M}$. Then
		\[ N_\mathcal{M}(t) = \frac{t^r}{(1-t)^r}\tilde{T}_{\mathcal{M}}(\nicefrac{1}{t},1). \]
	\end{theorem}
	
	\noindent
	This extends the results in \cite{DC-P-Box} and \cite{martino2017face}.
	
	\subsubsection*{The elliptic Tutte polynomial.}
	Let $\mathcal{E}$ be an elliptic arrangement \cite{MR3203648, MR3487239, MR1404924, MR1274092, MR3477330, BibbyGadish-elliptic} in $E(\Lambda)^d$, where $\Lambda$ is a lattice in $\mathbb{C}$ generated (as a group) by $1$ and by $w$.
	Let $T^e_\mathcal{E}(x,y)$ be the elliptic Tutte polynomial defined in the end of Section 4 of \cite{MR3487239}:
	\begin{equation}
	T^e_\mathcal{E}(x,y)=\sum_{S\subseteq [n]} m(S) (x-1)^{r'-\rk(S)}(y-1)^{|S|-\rk(S)},
	\end{equation}
	where $m(S)$ is the number of connected components of the intersections of elliptic hyperplanes $\cap_{i\in S} l_i$, $\rk(S)$ is the complex dimension of $\cap_{i\in S} l_i$, and $r'$ is the rank of $\mathcal{E}$.

	\noindent
	When $\operatorname{End}(E(\Lambda))=\mathbb{Z}$, there is a link between the combinatorics of elliptic arrangements and the combinatorics of $\mathbb{Z}$-matroids, see Example 2.1 in \cite{MR3487239}.
	While this connection is evident the meaning and the role of the multiplicity $m(S)$ is arithmetically different.
	
	We are able to provide an algebraic meaning to the multiplicity $m(S)$ in the case of elliptic arrangements of elliptic curves with complex multiplications; specifically, the elliptic Tutte polynomial $T^e_\mathcal{E}(x,y)$ is an evaluation of the Grothendieck-Tutte polynomial, see Proposition \ref{elliptic-tutte-is-our-tutte}.
	
	Finally, Bibby shows that whenever the isogenies group of the elliptic curve $E$ is $\mathbb{Z}$, then the Hilbert series of the model for the cohomology of the open complement $U$ of the elliptic arrangement in $E^d$ is a specialization of the Tutte.
	We can the extend of Bibby's Theorem to every elliptic arrangement, see Theorem \ref{thm-extension-bibby}
	

	\subsection*{Acknowledgements}
	The first author received support from the Scuola Superiore di Catania (Italy) and he is grateful to the Department of Mathematics of KTH for their hospitality. 
	The second author is currently supported by the Knut and Alice Wallenberg Foundation and by the Royal Swedish Academy of Science.
	
	We thank Christin Bibby, Emanuele Delucchi, Alex Fink, Roberto Pagaria, Dan Petersen, and Matteo Varbaro for several chats and email exchanges on the topics.

	\setcounter{section}{0}
	\setcounter{theorem}{0}
	\section{Preliminaries}
	
	\subsection{Simplicial posets}
	
	Let $(P,\leq)$ be a partially ordered set (poset). Let $a,b$ in $P$, then $b$ \emph{covers} $a$, written $a \triangleleft b$, if $a < b$ and there is no element $c \in P$ such that $a < c < b$. We define a partial order from the covering relation by declaring $a \leq b$ whenever there is an integer $n \in \mathbb{N}$ and elements $a_0,a_1,\dots,a_n \in P$ such that there is a \emph{chain} of covers
	\[ a = a_0 \triangleleft a_1 \triangleleft \dots \triangleleft a_{n-1} \triangleleft a_n = b. \]
	
	\noindent
	We allow our poset to be infinite, but we require to have finite chains.

	If $P$ has the minimum element we denote it by $\hat{0}$ and similarly we use $\hat{1}$ for the maximum element (if it exists).
	Let $a,b$ in $P$, we indicate with $a \vee b$ and with $a \wedge b$ respectively the least upper bound (\emph{join}) and the greatest lower bound (\emph{meet}) of $\{a,b\}$ (whenever they exist). A poset $P$ is a \emph{join-semilattice} if for every $a,b \in P$ the join $a \vee b$ exists; $P$ is a \emph{meet-semilattice} if for every $a,b \in P$ the meet $a \wedge b$ exists. The poset $P$ is a \emph{lattice} if it is both a meet-semilattice and a join-semilattice.
	
	A lattice $P$ is \emph{boolean} if it is \emph{distributive} (i.e. $\vee$ and $\wedge$ satisfy the distributive law), it has a minimum $\hat{0}$ and a maximum $\hat{1}$, and every element $a \in P$ has a (necessarily unique) \emph{complement}, that is an element $a' \in P$ such that $a \vee a' = \hat{1}$ and $a \wedge a' = \hat{0}$. 
	If $P$ is boolean, then all maximal chains in $P$ have the same length, that is the rank of $P$. (We remark that our chains are always finite.)
	A typical example of a boolean lattice is the power set $2^X$ ordered by inclusion, where $X$ is a set.
	
	A \emph{simplicial poset} $(P,\leq)$ is a poset with a minimum $\hat{0}$ and for every $a \in P$ the segment $[\hat{0},a] = \{ b \in P : \hat{0} \leq b \leq a \}$ is a boolean lattice. 
	We call \emph{rank} $\rk(a)$ of $a$ the rank of $[\hat{0},a]$.
	The maximum rank of all the elements of $P$ is the rank of the simplicial poset $P$, denoted with $\rk(P)$.
	
	\subsection{Face ring of simplicial posets}
	
	Following \cite{stanley1991f}, we now define a face ring for a simplicial poset $P$ (maybe infinite, but with finite rank). If $a,b$ in $P$, let $M(a,b)$ denote the set of \emph{minimal} upper bounds of $\{a,b\}$. In this work we require $M(a,b)$ to be finite. Let $\mathbb{K}$ be a field, we can associate with $P$ a polynomial ring $\mathbb{K}[x_a : a \in P]$. We set the degrees of $x_a$ following the rank of $a$: $\deg(x_a) = \rk(a)$ for all $a$ in $P$.
	
	\begin{definition}\label{def:Face-ring}
		The \emph{face ideal} of $P$ is the ideal of the polynomial ring $\mathbb{K}[x_a: a \in P]$ defined as
		\[ I_P = \left( x_{\hat{0}} - 1, \: x_ax_b - x_{a \wedge b} \left( \sum_{c \in M(a,b)} x_c \right)  : a,b \in P \right), \]
		where $x_{a \wedge b} = 0$ if $a \wedge b$ does not exists, and $\sum_{c \in M(a,b)}x_c = 0$ if $M(a,b) = \emptyset$. The \emph{face ring} of $P$ is the quotient
		\[ A_P = \frac{\mathbb{K}[x_a:a \in P]}{I_P}. \]
	\end{definition}
	
	The preceding definition generalizes the Stanley-Reisner ring of a simplicial complex. Here we briefly recall its definition.	Throughout this paper, we will set
	\[ [n] = \{1,2,\dots,n\}. \]
	An \emph{abstract simplicial complex} $\Delta$ on $n$ vertices is a collection of subsets of $[n]$ (called \emph{faces}) that is closed under taking subsets, that is if $A \subseteq B \in \Delta$ then $B \in \Delta$. 
	The \emph{Stanley-Reisner ideal} of $\Delta$ is the ideal of $\mathbb{K}[x_1,\dots,x_n]$ defined by $I_{\Delta} = (x_{i_1}\cdot\dots\cdot x_{i_r} : \{ i_1,\dots,i_r \} \notin \Delta)$. The \emph{Stanley-Reisner ring} (or \emph{face ring}) of $\Delta$ is the quotient $\mathbb{K}[\Delta] = \mathbb{K}[x_1,\dots,x_n]/I_{\Delta}$.
	
	If a simplicial poset $P$ is in addition a meet-semilattice, then $P$ is the \emph{face poset} (i.e. the poset of faces ordered by inclusion) of some simplicial complex, and one can check that the face ring of $P$ as a simplicial poset coincide with the face ring of the corresponding simplicial complex (see \cite{stanley1991f} or \cite[Chapter 3, Section 6]{stanley2007combinatorics}).
	
	\begin{example}\label{Example:simplicial poset}
		Consider the poset $P = \{ \hat{0},a,b_0,b_1,c_0,c_1,d_0,d_1 \}$ where $\hat{0}$ is the minimum, $a \leq c_i,d_i$ and $b_i \leq c_i,d_i$ for every $i \in \{0,1\}$.
		~\\
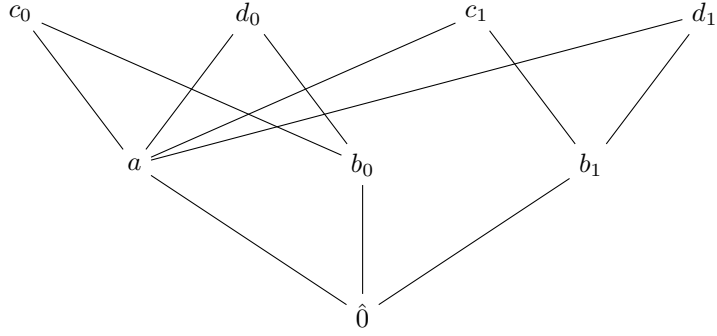
		\begin{center}
			\begin{tikzpicture}[scale=1, align=center]
			\node (00) at (0,0) {$\hat{0}$};
			\node (10) at (-3,2) {$a$};
			\node (20) at (0,2) {$b_0$};
			\node (21) at (3,2) {$b_1$};
			\node (120) at (-4.5,4) {$c_0$};
			\node (121) at (-1.5,4) {$d_0$};
			\node (122) at (1.5,4) {$c_1$};
			\node (123) at (4.5,4) {$d_1$};

			\draw (00) -- (10);
			\draw (00) -- (20);
			\draw (00) -- (21);
			
			\draw (10) -- (120);
			\draw (10) -- (121);
			\draw (10) -- (122);
			\draw (10) -- (123);
			
			\draw (20) -- (120);
			\draw (20) -- (121);
			\draw (21) -- (122);
			\draw (21) -- (123);
			
			\end{tikzpicture}
		\end{center}
		The poset $P$ is a simplicial poset, and its face ring is
		\[ A_P = \frac{\mathbb{K}[x_\alpha : \alpha \in P]}{I_P} \simeq \frac{\mathbb{K}[x_a,x_{b_0},x_{b_1},x_{c_0},x_{c_1},x_{d_0},x_{d_1}]}{
			\left(
			\begin{array}{l}			
			x_a x_{b_i} - (x_{c_i}+x_{d_i}), x_{b_0} x_{b_1}, \\
			x_{c_i} x_{d_j}, x_{c_0} x_{c_1}, x_{d_0} x_{d_1}, \\
			x_{b_i} x_{c_{\overline{i}}}, x_{b_i} x_{d_{\overline{i}}}
			\end{array}
			:
			\begin{array}{l}
			i,j \in \{0,1\} \\
			\overline{i} = 1-i
			\end{array}
			\right)
		} \]
	\end{example}
	
	\subsection{The $f$-vector, $h$-vector, and the Hilbert series}
	Whenever $P$ is finite we count with $f_i$ the number of elements of $P$ of rank $i+1$. There exists a maximum integer $r \in \mathbb{N}$ (the rank of $P$) such that $f_{r-1} \neq 0$. The \emph{$f$-vector} of $P$ is the vector $(f_{-1},f_0,\dots,f_{r-1})$. The \emph{$h$-vector} of $P$ is the vector $(h_0,\dots,h_r)$ defined by the formula $\sum_{i=0}^r f_{i-1}(t-1)^{r-i} = \sum_{i=0}^r h_i t^{r-i}$.
	
	Let $A$ be a finitely generated $\mathbb{N}$-graded $\mathbb{K}$-algebra, and let $N$ be a finitely generated graded $A$-module. Denote by $N_i$ the homogeneous part of degree $i$. Since $N$ is finitely generated, $N_i$ is a finitely generated $\mathbb{K}$-vector space, and we denote its dimension with $\dim_{\mathbb{K}}N_i$. The \emph{Hilbert series} of $N$ is
	\[ N(t) = \sum_{i \in \mathbb{N}} \dim_\mathbb{K}(N_i)t^i. \]
	
	Under the specific grading provided by $\deg(x_a) = \rk(a)$ for all $a$ in $P$, the face ring $A_P$ is a graded $\mathbb{K}[x_a: a \in P]$-module. If $P$ is finite, then $\mathbb{K}[x_a: a \in P]$ is a finitely generated $\mathbb{N}$-graded $\mathbb{K}$-algebra, and $A_P$ is a finitely generated graded module. Its Hilbert series is related to the $h$-vector by the following result.
	
	\begin{theorem}[Stanley, \cite{stanley1991f}]\label{hilbert series and h vector}
		Let $P$ be a finite simplicial poset of rank $r$ and h-vector $(h_0,\dots,h_r)$. With the grading of $A_P$ just defined, we have
		\[ A_P(t) = \frac{ h_0 + h_1t+\dots+h_rt^r }{(1-t)^r}. \]
	\end{theorem}
	
	\subsection{Matroids}\label{matroids}
	
	A \emph{matroid} on the ground set $[n]$ is a collection $\mathcal{I}$ of subsets of $[n]$ (called \emph{independent sets}), such that
	\begin{enumerate}[label=(\emph{I}\arabic*)]
		\item $\emptyset \in \mathcal{I}$.
		\item $A \subseteq B \in \mathcal{I} \Rightarrow A \in \mathcal{I}$.
		\item $A,B \in \mathcal{I}$, $|A| < |B|$ $\Rightarrow$ $\exists b \in B \setminus A: A \cup \{b\} \in \mathcal{I}$.
	\end{enumerate}
	The first two axioms make $\mathcal{I}$ into a (non-empty) simplicial complex. Axiom $I3$ is sometimes referred as \emph{independent set exchange property} (or \emph{independence augmentation axiom.}). Let $\mathcal{I}$ be a matroid on the ground set $[n]$, and let $A \subseteq [n]$. All maximal independent subsets of $A$ have the same cardinality, called the \emph{rank} $\rk(A)$ of $A$, whereas the \emph{corank} of $A$ is $\cork(A) = \rk([n]) - \rk(A)$. A matroid can be equivalently defined by assignining the (co)rank of the subsets of $[n]$ (see \cite[Corollary 1.3.4]{oxley2006matroid}).
	
	In \cite{fink2016matroids}, Fink and Moci generalize the notion of matroid, by giving the definition of matroid of modules over a (commutative) ring $R$. Here we recall the definitions and a few notations related. Let $R$ be a commutative ring and denote with $R$-mod the category of finitely generated $R$-modules.
	
	\begin{definition}\label{definition matroid over ring}
		A \emph{matroid of modules over $R$} on the ground set $[n]$ is a function
		\[ \mathcal{M}: 2^{[n]} \rightarrow R\text{-mod} \]
		such that for every $A \subseteq [n]$ and $b,c \in [n] \setminus A$, there exist $x,y \in \mathcal{M}(A)$ such that
		\begin{align*}
		\mathcal{M}(A \cup \{b\}) & \simeq \mathcal{M}(A) / (x) \\
		\mathcal{M}(A \cup \{c\}) & \simeq \mathcal{M}(A) / (y) \\
		\mathcal{M}(A \cup \{b,c\}) & \simeq \mathcal{M}(A) / (x,y)
		\end{align*}
		(note that the choice of $x$ and $y$ depends on both $b$ and $c$).
	\end{definition}
	
	A matroid of modules $\mathcal{M}$ over $R$ is \emph{essential} if no nontrivial projective module is a direct summand of $\mathcal{M}([n])$. In \cite[Proposition 2.6]{fink2016matroids} it was proved that essential matroids over a field $\mathbb{K}$ are equivalent to classical matroids. To avoid confusion, from now on with "matroid" (or "$R$-matroid"), we mean matroid of modules over a (commutative) ring $R$, and with "classical matroid" we refer to the independent sets definition. 
	Given an essential $\mathbb{K}$-matroid $\mathcal{M}$, the corresponding classical matroid can be recovered by assigning the corank: $\cork(A) = \dim_{\mathbb{K}} \mathcal{M}(A)$. 
	%
	
	An $R$-matroid $\mathcal{M}$ on the ground set $[n]$ is \emph{realizable} (or \emph{representable}) if there exists a map $\psi:[n] \rightarrow \mathcal{M}(\emptyset)$ such that
	\[ \mathcal{M}(A) = \mathcal{M}(\emptyset)/(\psi(i):i \in A) \quad \text{for all } A \subseteq [n]. \]
	In this case, $\psi$ is a \emph{realization} of $\mathcal{M}$. 
	We often mention \emph{realized} matroid and we mean a realizable matroid together with a given specific realization.
	Note that an essential $\mathbb{K}$-matroid is realizable if and only if its associated classical matroid is realizable over $\mathbb{K}$.

	Let $\mathcal{M}$ and $\mathcal{M}'$ be two $R$-matroids on respective ground sets $[n]$ and $[m]$. Their \emph{direct sum} $\mathcal{M} \oplus \mathcal{M}'$ is an $R$-matroid on the ground set $[n] \amalg [m]$ defined by
	\[ (\mathcal{M} \oplus \mathcal{M}')(A \amalg A') = \mathcal{M}(A) \oplus \mathcal{M}'(A'). \]
	Let $\mathcal{M}$ be an $R$-matroid on the ground set $[n]$. If $i \in [n]$, the \emph{deletion} $\mathcal{M} \setminus i$ and the \emph{contraction} $\mathcal{M} / i$ of $i$ in $\mathcal{M}$ are $R$-matroids on the ground set $[n] \setminus i$ defined respectively by
	\begin{eqnarray*}
		(\mathcal{M} \setminus i)(A) &=& \mathcal{M}(A), \\
		(\mathcal{M} / i)(A) &=& \mathcal{M}(A \cup \{i\}).
	\end{eqnarray*}
	If $N$ is an $R$-module, the \emph{empty matroid} $\mathcal{M}$ for $N$ is the matroid on the ground set $\emptyset$ that assigns $N$ to $\emptyset$. If the module $N$ is projective, then $\mathcal{M}$ is a \emph{projective empty matroid}. From \cite[Lemma 2.5]{fink2016matroids}, every $R$-matroid $\mathcal{M}$ is the direct sum of an essential $R$-matroid $\mathcal{M}_E$ and a projective empty matroid $\mathcal{M}_P$:
	\begin{equation}\label{eq: essential plus empty}
	\mathcal{M} = \mathcal{M}_E \oplus \mathcal{M}_P.
	\end{equation}
	Let $S$ be a ring and $R \rightarrow S$ an homomorphism of rings. If $\mathcal{M}$ is an $R$-matroid, define
	\[ (\mathcal{M} \otimes_R S) (A) = \mathcal{M}(A) \otimes_R S \quad \text{for all } A \subseteq [n]. \]
	$\mathcal{M} \otimes_R S$ is a matroid over $S$ on the same ground set of $\mathcal{M}$ (\cite[Proposition 2.7]{fink2016matroids}). In particular, let $R$ be a domain and $Q(R)$ be its field of fractions, if we decompose $\mathcal{M}$ as in (\ref{eq: essential plus empty}), then $\mathcal{M}_E \otimes Q(R)$ is called the \emph{generic matroid}. Since the generic matroid is an essential matroid over the field $Q(R)$, we can consider its associated classical matroid. 
	Therefore, from an $R$-matroid $\mathcal{M}$, we can associate the classical matroid $\mathcal{M}_E \otimes Q(R)$. We denote by $\Delta \mathcal{M}$, the simplicial complex of independent sets of $\mathcal{M}_E \otimes Q(R)$. When we use the terminology of classical matroids (i.e. independent sets, rank, \dots) for an $R$-matroid, we will always refer to its generic matroid.

	\subsection{The Tutte polynomial for $R$-matroids}\label{sec:preliminary-Tutte}

	Let $L_0(R\text{-mod})$ be the commutative ring free generated as a group by the isomorphism classes of finitely generated $R$-modules $[N]$ with product $[N][N']=[N\oplus N']$, 
	$L_0(R\text{-mod})$ is a Grothendieck style ring, but we avoid the classical relation $[A]-[B]+[C]=0$ for any short exact sequence $0\rightarrow A \rightarrow B\rightarrow C \rightarrow 0$. (We note that the neutral element $1$ for the multiplication is the class of the zero module.) 
	
	\begin{remark}
		In \cite{fink2016matroids}, Fink and Moci use $\mathbb{Z}[R\text{-mod}]$ for this ring, $L_0(R\text{-mod})$. They denote by $u^{N}$ the class of the $R$-module $N$.
		This object was also used in \cite{MR3657413, MR3423448}.
	\end{remark}
	
	In Section 7 of \cite{fink2016matroids}, Fink and Moci have defined the Tutte-Grothendieck polynomial as an element of the Tutte-Grothendieck ring of an $R$-matroid, which we here denote by K($R$-Mat), essentially following \cite{MR0330004}. They have shown that whenever $R$ is a Dedekind Domain the Tutte-Grothendieck ring injects into $L_0(R\text{-mod})\otimes_{\mathbb{Z}} L_0(R\text{-mod})$ and, specifically, the matroid $\mathcal{M}$ is sent to a Tutte-Grothendieck polynomial style elements:
	\begin{equation*}
	GT_{\mathcal{M}}=\sum_{A\subseteq E} [\mathcal{M}(A)]\otimes[\mathcal{M}^*(E\setminus A)]\in L_0(R\text{-mod})\otimes_{\mathbb{Z}} L_0(R\text{-mod}),
	\end{equation*}
	where the sum runs over all possible subsets $A$ of the groundset $E$ of the matroid $\mathcal{M}$. We refer to Section 4 of \cite{fink2016matroids}, for a precise definition of the dual matroid $\mathcal{M}^*$ over a Dedekind Domain $R$; the duality for matroids over $\mathbb{Z}$ is provided in Section 7 of \cite{MR2989987}. 
	
	If $R$ is a field, then every module is free and  $L_0(R\text{-mod})$ is isomorphic to $\mathbb{Z}[x]$, by assigning $x^d$ to $[R^d]$.
	%
	%
	Hence, one readily gets the classical form for the Tutte polynomial for a classical matroid $\mathcal{M}$ of rank $r$
	\begin{equation*}
	T_{\mathcal{M}}(x,y)=\sum_{A\subseteq [n]} (x-1)^{r-\rk(A)}(y-1)^{|A|-\rk(A)}\in \mathbb{Z}[x, y].
	\end{equation*}
	
	\begin{remark}
		Note that if $\mathbf{k}$ is a field, $K_0(\mathbf{k}\text{-mod})$ differs from $L_0(\mathbf{k}\text{-mod})$,as the first is isomorphic to $\mathbb{Z}$.
	\end{remark}
	
	\noindent
	In a similar fashion, for $\mathbb{Z}$-matroids, one can evaluate (see Section 7.1 of \cite{fink2016matroids}) the $GT_{\mathcal{M}}$ to the arithmetic Tutte polynomial \cite{MR3240933, MR2846363}:    \begin{equation}\label{eq-Tutte-classical}
	T_{\mathcal{M}}(x,y)=\sum_{A\subseteq [n]} m(A) (x-1)^{r-\rk(A)}(y-1)^{|A|-\rk(A)}\in \mathbb{Z}[x, y].
	\end{equation}
	
	\noindent
	In Section \ref{Sec-Tutte-General}, we will show concretely what the Tutte-Grothendieck polynomial of a matroid over a domain $R$ looks like, and our results will generalize the formula in (\ref{eq-Tutte-classical}).

	\subsection{Modules over a (Dedekind) domain}
	
	Let $R$ be a domain, $Q(R)$ its field of fractions and $M$ an $R$-module. 
	An element of $x$ of $M$ is a \emph{torsion element} if it has non trivial annihilator, i.e. $\Ann(x)=\{r\in R: rx=0\} \neq 0$.
	The set of torsion elements is the \emph{torsion part} of $M$:
	\[ \tor{M} = \{ x \in M : \exists r \in R, r \neq 0 \: rx = 0 \}. \]
	Since $R$ is a domain, $\tor{M}$ is a submodule of $M$. If $\tor{M} = 0$, then $M$ is \emph{torsion-free}.
	
	\begin{lemma}\label{injection of torsion}
		If $x \in M$ is not a torsion element and $\pi:M \rightarrow M/(x)$ is the natural surjection, then $\pi_{|\tor{M}}$ is injective.
	\end{lemma}
	\begin{proof}
		Let $y \in \ker \pi$ then either $y=0$ or $y$ is not a torsion element. In fact, if $y \in \ker \pi = (x)$ and $y \neq 0$, then there exists $r \in R$, $r \neq 0$ such that $y = rx$, so
		\[ \Ann(y) \subseteq \Ann(x) = 0. \qedhere \]
	\end{proof}
	
	\begin{lemma}\label{torsion free tensor product}
		If $M$ is torsion-free, then $M \otimes_R Q(R) = 0$ implies $M = 0$.
	\end{lemma}
	\begin{proof}
		Let $S = R \setminus \{0\}$ the multiplicative set of $R$. Since $M$ is torsion-free, $\varphi:M \rightarrow S^{-1}M$ defined by $\varphi(m) = m/1$ is injective. Hence
		\[ M \subseteq S^{-1}M \simeq M \otimes_R Q(R) = 0. 
		\qedhere \]
	\end{proof}
	
	\subsubsection{Structure Theorem for modules over a Dedekind Domain}
	We now briefly recall the definition of a Dedekind domain. 
	An $R$-submodule $N$ of $Q(R)$ is a \emph{fractional ideal} of $R$ if $xN \subseteq R$ for some non zero $x \in R$. A fractional ideal $N$ is \emph{invertible} if there exists a fractional ideal $N'$ such that $NN' = R$. The domain $R$ is \emph{Dedekind} if every fractional ideal is invertible. The last condition is equivalent to requiring that every localization $R_{\mathfrak{p}}$ of $R$ at a prime ideal $\mathfrak{p}$ is a Discrete Valuation Ring. For further reading we suggest \cite[Chapter 9]{atiyah2018introduction}.
	
	Let $R$ be a Dedekind domain, and $M$ an $R$-module. The quotient $M/\tor(M)$ is a projective module and, if $M$ is finitely generated, then
	\[ M \simeq M/\tor(M) \oplus \tor(M), \]
	see for instance \cite[Proposition 3.3]{fink2016matroids}.

	\section{The Grothendieck-Tutte polynomial}\label{Sec-Tutte-General}
	
	In this section, $R$ is a domain with field of fractions $Q(R)$, and $\mathcal{M}$ is an $R$-matroid of rank $r$ on the ground set $[n]$.
	%
	%
	%
	We recall our notations in Section \ref{matroids}:
	$\mathcal{M} \otimes Q(R)$ is a matroid over the field $Q(R)$ (on the same ground set $[n]$), so its essential part $\mathcal{M}_E$ is a classical matroid, and we identify with $\Delta \mathcal{M}$ the associated simplicial complex of independent sets.
	For every subset $A$ of $[n]$ the torsion part of $\mathcal{M}(A)$ is $\torM(A)$, and the rank (associated to the generic matroid) is $\rk_{\mathcal{M}}(A)$.
	
	\noindent
	Moreover, $L_0(R\text{-mod})$ is the commutative ring generated by the isomorphism classes of finitely generated $R$-modules $[N]$ under the relation $[N\oplus N']=[N][N']$, which also defines the product operation in this ring; we wrote more about it in the prelimiary Section \ref{sec:preliminary-Tutte}.
	Let $\vee$ denote the application of the controvariant functor $\Hom(-,\nicefrac{Q(R)}{R})$, so that $\tor (A)^{\vee} = \Hom(\tor (A),\nicefrac{Q(R)}{R})$.
	
	\begin{definition}\label{Definition:GT-domain}
		We define the Tutte polynomial for the matroid $\mathcal{M}$ over a domain $R$ as the following polynomial with coefficients in $L_0(R\text{-mod})$:
		\[ 
		T_{\mathcal{M}}(x,y) = \sum_{A \subseteq [n]} [\tor(A)^\vee] (x-1)^{r-\rk(A)}(y-1)^{|A|-\rk(A)} \in L_0(R\text{-mod})[x, y]. 
		\]
	\end{definition}
	
	\noindent
	We could decompose the sum of the previous definition in the following way:
	\begin{equation}\label{eq:Tutte decomposition}
	\begin{split}
	T_{\mathcal{M}}(x,y) & = \sum_{A \subseteq [n] \setminus i} [\tor(A)^\vee] (x-1)^{r-\rk(A)}(y-1)^{|A|-\rk(A)} + \\
	& + \sum_{A \subseteq [n] \setminus i} [\tor(A \cup \{i\})^\vee] (x-1)^{r-\rk(A \cup \{i\})}(y-1)^{|A \cup \{i\}|-\rk(A \cup \{i\})}
	\end{split}
	\end{equation}
	The sum runs over all subsets $A$ in $[n] \setminus i$, where the latter means $[n] \setminus \{i\}$. For sake of space, we keep this abuse of notation all along this section. 
	
	\begin{theorem}\label{tutte formula}
		Let $R$ be a domain and $\mathcal{M}$ be an $R$-matroid of rank $r$ on the ground set $[n]$. If $i \in [n]$ is not a loop or coloop, then
		\[ T_{\mathcal{M}}(x,y) = T_{\mathcal{M} \setminus i}(x,y) + T_{\mathcal{M} / i}(x,y). \]
	\end{theorem}
	\begin{proof}
		Since $i \in [n]$ is not a loop or coloop, as for classical matroids, for every $A \subseteq [n] \setminus i$ we have
		\begin{align*}
		\rk_{\mathcal{M}}(A)& = \rk_{\mathcal{M} \setminus i}(A), \\
		\rk_{\mathcal{M}}(A \cup \{i\}) &= \rk_{\mathcal{M} / i}(A) +1.
		\end{align*}
		Futher, by definition
		\begin{align*}
		\tor_{\mathcal{M} \setminus i}(A) &= \torM(A), \\
		\tor_{\mathcal{M} / i}(A) &= \torM(A \cup \{i\}).
		\end{align*}
		Finally by substituting in (\ref{eq:Tutte decomposition}) we obtain
		\[ T_{\mathcal{M}}(x,y) = T_{\mathcal{M} \setminus i}(x,y) + T_{\mathcal{M} / i}(x,y). \qedhere \]
	\end{proof}
	
	\begin{proposition}\label{lemma torsion-free loop}
		If $\mathcal{M}(\emptyset)$ is torsion-free and $i \in [n]$ is a loop, then for every $A \subseteq [n]\setminus i$ we have $\mathcal{M}(A) \simeq \mathcal{M}(A \cup \{i\})$, in particular $\tor(A) \simeq \tor(A \cup \{i\})$.
	\end{proposition}
	\begin{proof}
		First, write $\mathcal{M} = \mathcal{M}_E \oplus \mathcal{M}_P$, where $\mathcal{M}_E$ is an essential matroid and $\mathcal{M}_P$ is a projective empty matroid with $\mathcal{M}_P(\emptyset) = P$.
		We proceed by induction on the cardinality of $A$. For the base case, since $i \in [n]$ is a loop, we have
		\[ \mathcal{M}_E(\emptyset) \otimes Q(R) \simeq \mathcal{M}_E(i) \otimes Q(R). \]
		By definition of matroid, $\mathcal{M}_E(i) \simeq \mathcal{M}_E(\emptyset) / (x)$ for some $x \in \mathcal{M}_E(\emptyset)$. Since $Q(R)$ is a flat $R$-module, it follows
		\[ \mathcal{M}_E(\emptyset) \otimes Q(R) \simeq \left(\frac{\mathcal{M}_E(\emptyset)}{(x)}\right) \otimes Q(R) \simeq \frac{\mathcal{M}_E(\emptyset) \otimes Q(R)}{(x) \otimes Q(R)}, \]
		from \cite[Corollary 4.4]{eisenbud2013commutative} the natural projection
		\[ \mathcal{M}_E(\emptyset) \otimes Q(R) \twoheadrightarrow \frac{\mathcal{M}_E(\emptyset) \otimes Q(R)}{(x) \otimes Q(R)} \]
		is an isomorphism, so its kernel is zero $(x) \otimes Q(R) = 0$. Since $\mathcal{M}(\emptyset)$ is torsion-free, then also $(x)$ is torsion-free, so from Lemma \ref{torsion free tensor product} it follows that $(x) = 0$ implies $x=0$. Hence $\mathcal{M}_E(\emptyset) \simeq \mathcal{M}_E(i)$, so
		\[ \mathcal{M}(\emptyset) = \mathcal{M}_E(\emptyset) \oplus P \simeq \mathcal{M}_E(i) \oplus P = \mathcal{M}(i). \]
		Now for the inductive step, let $A \subseteq [n] \setminus i$ and $a \in A$. By definition of matroid there exist $y,z \in \mathcal{M}(A \setminus a)$ such that
		\begin{align*}
		\mathcal{M}(A) & \simeq \mathcal{M}(A \setminus a)/(y) \\
		\mathcal{M}((A \cup \{i\}) \setminus a) & \simeq \mathcal{M}(A \setminus a) / (z) \\
		\mathcal{M}(A \cup \{i\}) & \simeq \mathcal{M}(A \setminus a)/ (y,z).
		\end{align*}
		By the inductive hypothesis, we have $\mathcal{M}(A \setminus a) \simeq \mathcal{M}(A \setminus a) / (z)$. Again from \cite[Corollary 4.4]{eisenbud2013commutative} the natural projection $\mathcal{M}(A \setminus a) \twoheadrightarrow \mathcal{M}(A \setminus a)/(z)$ is an isomorphism, so its kernel is zero, so $z = 0$. Now it easily follows $\mathcal{M}(A) \simeq \mathcal{M}(A \cup \{i\})$.
	\end{proof}
	
	\begin{corollary}\label{torsion-free loop}
		If $\mathcal{M}(\emptyset)$ is torsion-free and $i \in [n]$ is a loop, then
		\[ T_{\mathcal{M}}(x,y) = y T_{\mathcal{M} \setminus i}(x,y) \]
	\end{corollary}
	\begin{proof}
		From Proposition \ref{lemma torsion-free loop}, $[\tor(A)^\vee] = [\tor(A \cup \{i\})^\vee]$, further
		\[ \rk_{\mathcal{M}}(A \cup \{i\}) = \rk_{\mathcal{M}}(A) = \rk_{\mathcal{M} \setminus i}(A). \]
		Hence, by substituting in (\ref{eq:Tutte decomposition}) we obtain
		\[ T_{\mathcal{M}}(x,y) = T_{\mathcal{M} \setminus i}(x,y) + (y-1)T_{\mathcal{M} \setminus i}(x,y) = y T_{\mathcal{M} \setminus i}(x,y). \qedhere  \]
	\end{proof}


	\begin{proposition}\label{lemma no torsion coloop}
		If $\mathcal{M}([n]) = 0$ and $i \in [n]$ is a coloop, then for every $A \subseteq [n] \setminus i$ we have $\mathcal{M}(A) \simeq \mathcal{M}(A \cup \{i\}) \oplus R$, in particular $\tor(A) \simeq \tor(A \cup \{i\})$.
	\end{proposition}
	\begin{proof}
		Note that $\mathcal{M}$ is an essential matroid. 
		We want to prove the statement by induction on the cocardinality, $n-|A|$. For the base case, consider $[n]\setminus i$, by definition of matroid, we have $\mathcal{M}([n] \setminus i) / (x) \simeq \mathcal{M}([n])  = 0$ for some $x \in \mathcal{M}([n] \setminus i)$. Since $i \in [n]$ is a coloop, we have
		\[ Q(R) \simeq \mathcal{M}([n] \setminus i) \otimes Q(R) \simeq (x) \otimes Q(R), \]
		so $x \in \mathcal{M}([n] \setminus i)$ is not a torsion element, therefore $R \simeq (x) \simeq \mathcal{M}([n] \setminus i)$. For the inductive step, let $A \subseteq [n] \setminus i$ and $a \in A$. By definition of matroid, there exist $y,x \in \mathcal{M}(A \setminus a)$ such that
		\begin{align*}
		\mathcal{M}(A) & \simeq \mathcal{M}(A \setminus a)/(y) \\
		\mathcal{M}((A \cup \{i\}) \setminus a) & \simeq \mathcal{M}(A \setminus a) / (z) \\
		\mathcal{M}(A \cup \{i\}) & \simeq \mathcal{M}(A \setminus a)/ (y,z).
		\end{align*}
		By the inductive hypothesis, $\mathcal{M}(A) \simeq \mathcal{M}(A \cup \{i\}) \oplus R \simeq \mathcal{M}(A)/(\overline{z}) \oplus R$, in particular, $\overline{z}$ (and also $z$) is not a torsion element. Otherwise, tensoring by $Q(R)$, which is a flat $R$-module, we would have
		\[ \mathcal{M}(A) \otimes Q(R) \simeq \frac{\mathcal{M}(A) \otimes Q(R)}{(\overline{z}) \otimes Q(R)} \oplus Q(R) \simeq \big(\mathcal{M}(A) \otimes Q(R)\big) \oplus Q(R), \]
		which is a contradiction. Hence $(z) \simeq (\overline{z}) \simeq R$ and $\mathcal{M}(A) \simeq \mathcal{M}(A) /(\overline{z}) \oplus (\overline{z})$, so the second row of the following diagram splits
		\[ \begin{tikzcd}
		0 \arrow[r] & (z) \arrow[r] \arrow[d, leftrightarrow, "\simeq"] & \mathcal{M}(A \setminus a) \arrow[r] \arrow[d, twoheadrightarrow] & \mathcal{M}(A \setminus a) / (z) \arrow[r] \arrow[d, twoheadrightarrow] & 0 \\
		0 \arrow[r] & (\overline{z}) \arrow[r, shift right=0.5ex] & \mathcal{M}(A) \arrow[r] \arrow[l, dashed, shift right=0.5ex] & \mathcal{M}(A)/(\overline{z}) \arrow[r] & 0
		\end{tikzcd} \]
		further, since the composition $(z) \hookrightarrow \mathcal{M}(A \setminus a) \twoheadrightarrow \mathcal{M}(A) \dashrightarrow (\overline{z}) \simeq (z)$ is the identity on $(z)$, also the first row splits, so
		\[ \mathcal{M}(A \setminus a) \simeq \mathcal{M}(A \setminus a) / (z) \oplus (z) \simeq \mathcal{M}((A \cup \{i\}) \setminus a) \oplus R. \qedhere \]
	\end{proof}
	
	\begin{remark}
		Whenever the dual matroid exists $\mathcal{M}^*$, the previous result is obtained with the less restrictive hypothesis of $\mathcal{M}([n])=0$ being torsion free by observing that a coloop for $\mathcal{M}$ is actually a loop for $\mathcal{M}^*$.
	\end{remark}
	
	\begin{corollary}\label{no torsion coloop}
		If $\mathcal{M}([n]) = 0$ and $i \in [n]$ is a coloop, then
		\[ T_{\mathcal{M}}(x,y) = x T_{\mathcal{M} / i}(x,y) \]
	\end{corollary}
	\begin{proof}
		From Proposition \ref{lemma no torsion coloop}, $[\tor(A)^\vee] = [\tor(A \cup \{i\})^\vee]$, further
		\[ \rk_{\mathcal{M}}(A \cup \{i\}) = \rk_{\mathcal{M}}(A)+1 = \rk_{\mathcal{M} / i}(A)+1. \]
		Hence, by substituting in (\ref{eq:Tutte decomposition}) we obtain
		\[ T_{\mathcal{M}}(x,y)  = (x-1)T_{\mathcal{M} / i}(x,y) + T_{\mathcal{M} / i}(x,y) = x T_{\mathcal{M} / i}(x,y). \qedhere  \]
	\end{proof}
	
	Finally, from Theorem \ref{tutte formula} and Corollaries \ref{torsion free tensor product} and \ref{no torsion coloop} we obtain a generalization of a well-known recursive formula for the Tutte polynomial for classical matroids.
	
	\begin{theorem}
		Let $R$ be a domain and $\mathcal{M}$ be an $R$-matroid of rank $r$ on the ground set $[n]$. If $\mathcal{M}(\emptyset)$ is torsion-free and if $\mathcal{M}([n]) = 0$, then
		\[ T_{\mathcal{M}}(x,y) = \begin{cases}
		y T_{\mathcal{M} \setminus i}(x,y) & \text{if $i$ is a loop,} \\
		x T_{\mathcal{M} / i}(x,y) & \text{if $i$ is a coloop,} \\
		T_{\mathcal{M} \setminus i}(x,y) + T_{\mathcal{M} / i}(x,y) & \text{otherwise.}
		\end{cases} \]
	\end{theorem}
	

	\section{The Grothendieck $f$-vector}\label{sec:f-vector}
	Another combinatorial object that we are going to generalize for matroids over a domain is the $f$-vector. This should \emph{count} then number of elements of the matroid of a specific rank. 
	For matroids over a domain, this enumeration should run over the torsion submodules similarly as for the Tutte polynomial.
	
	
	\begin{definition}\label{def:f-vector}
		Let $\mathcal{M}$ be a matroid over a domain $R$ and let $L_0(R\text{-mod})$ be the ring defined in Section \ref{sec:preliminary-Tutte}. 
		Recall that $\Delta \mathcal{M}$ is the independent complex of the generic matroid of $\mathcal{M}$.
		We define 
		\[
		f_{i-1}=\sum_{\substack{A\in \Delta \mathcal{M}, \\ |A|=i}} [\tor (A)^{\vee}]\in L_0(R\text{-mod})
		\]
		We refer to the vector $(f_{-1}, f_{0}, \dots, f_{r-1})$ in $L_0(R\text{-mod})^{r+1}$ as the Grothendieck $f$-vector of the matroid  $\mathcal{M}$.
	\end{definition}
	
	We are going to see in Section \ref{sec:poset of torsions}, that the definition of Grothendieck $f$-vector is inspired by the $f$-vector of a certain simplicial poset.
	
	One can also define the prototype of the $h$-vector by using the classical relation:
	\[
	\sum_{i=0}^r f_{i-1} (t-1)^{r-i} = \sum_{i=0}^r h_i t^{r-i}\in L_0(R\text{-mod})[t].
	\]
	The left hand side is called the $f$-polynomial and denoted by $f_{\mathcal{M}}(t)$.

	As in the classical case and in the arithmetic case, the Tutte polynomial does save a lot of combinatorial information. For instance, we can read back the $f$-vector of the matroid $\mathcal{M}$.
	
	\begin{theorem}\label{thm:tutte-f-vector}
		Let $R$ be a domain and $\mathcal{M}$ be an $R$-matroid of rank $r$ on the ground set $[n]$. Then, the $f$-polynomial $f_{\mathcal{M}}(t)$ is a specialization of the Tutte polynomial, i.e.
		\[ 
		T_{\mathcal{M}}(t,1) = \sum_{i=0}^r f_{i-1} (t-1)^{r-i};
		\]
		Moreover, 
		\begin{equation}\label{eq:almost-hilbert-series}
		\frac{t^r}{(1-t)^r} T_{\mathcal{M}}\left(\nicefrac{1}{t},1\right) = \sum_{i=0}^r f_{i-1} \frac{t^i}{(1-t)^i}.
		\end{equation}
	\end{theorem}
	\begin{proof}
		We simply evaluate the Tutte polynomial at (1,t) and we obtain readily that:
		\[ \begin{split}
		T_{\mathcal{M}}(t,1) & = \sum_{A\in \Delta \mathcal{M}} [\tor(A)^{\vee}](t-1)^{r-|A|} = \sum_{i=0}^r \left( \sum_{\substack{A\in \Delta \mathcal{M}, \\ |A|=i}} [\tor(A)^{\vee}] \right) (t-1)^{r-i} = \\
		& = \sum_{i=0}^r f_{i-1} (t-1)^{r-i}.
		\end{split} 
		\]
		For the second statement, we substitute in the previous result $\nicefrac{1}{t}$ instead of $t$.
	\end{proof}
	
	\begin{remark}
		The previous proposition shows how far the connection between the enumerative Combinatorics and Algebra can be pushed over a generic domain $R$.
		
		\noindent
		In the introduction we have explained that, in literature, the proper statement (see Section A.3 in \cite{DC-P-Box}) would link $\frac{t^r}{(1-t)^r} T_{\mathcal{M}^*}\left(1, \nicefrac{1}{t}\right)$ to the Hilbert series of the face module of a classical matroid or a $\mathbb{Z}$-matroid. 
		Indeed, in the classical case (when $R$ is a field or when $R$ is the ring of integers) the right hand side of (\ref{eq:almost-hilbert-series}) is the Hilbert series $N_{\Gr \mathcal{M}}(t)$ of the poset of torsion $\Gr \mathcal{M}$.
		We are going to show that so it is for many other case of matroids over a domain in Section \ref{Sec:ring-of-integers}.
	\end{remark}

	\section{The poset of torsions}\label{sec:poset of torsions}

	In this section, $R$ is a domain with field of fractions $Q(R)$, and $\mathcal{M}$ is a realizable $R$-matroid on the ground set $[n]$. Fix a realization $\psi:[n] \rightarrow \mathcal{M}(\emptyset)$ of $\mathcal{M}$. 
	Recall that $\Delta \mathcal{M}$ is the independent complex of the generic matroid of $\mathcal{M}$.

	Let $A$ be a subset of $[n]$ and $b \in [n] \setminus A$ such that $A \cup \{b\} \in \Delta \mathcal{M}$. Then, of course, $A \in \Delta \mathcal{M}$ and $\psi(b) \in \mathcal{M}(\emptyset)$ is not a torsion element. 
	By Definition \ref{definition matroid over ring} there is the quotient map
	\[ \overline{\pi_{A,b}}:\mathcal{M}(A) \rightarrow \mathcal{M}(A \cup \{b\}) \simeq \mathcal{M}(A) / (\psi(b)). \]
	
	\noindent
	We will denote the torsion part of $\mathcal{M}(A)$ by
	$\torM{A}$, when there is no ambiguity, simply by $\tor{A}$. If we restrict $\overline{\pi_{A,b}}$ to $\tor{A}$, from Lemma \ref{injection of torsion}, we obtain an injective map denoted by
	\[ \pi_{A,b}: \tor (A) \rightarrow \tor (A \cup \{b\}). \]
	Recall that $\vee$ denote the application of the controvariant functor $\Hom(-,\nicefrac{Q(R)}{R})$, so that $\tor (A)^{\vee} = \Hom(\tor (A),\nicefrac{Q(R)}{R})$. We obtain the sujective map
	\[ \pi_{A,b}^{\vee}:\tor (A \cup \{b\})^{\vee} \twoheadrightarrow \tor (A)^{\vee}. \]
	
	\begin{remark}
		Since $R$ is a domain, $\tor \mathcal{M}(A)$ is $\operatorname{Tor}_1^{R}(\mathcal{M}(A), \nicefrac{Q(R)}{R})$, see Exercise 3.1.2 in \cite{MR1269324}.
		Moreover, $\tor (A)^{\vee} = \Hom(\tor (A),\nicefrac{Q(R)}{R})$ is usually called the \emph{full dual} of the $R$-module $\tor (A)^{\vee}$ and, by definition, this can also be written as  $\operatorname{Ext}_R^{0}(\tor \mathcal{M}(A), \nicefrac{Q(R)}{R})$. 
	\end{remark}

	The next definition generalizes the poset of torsions introduced in \cite{martino2017face} for matroids over $\mathbb{Z}$ by the second author.
	
	\begin{definition}
		Let $\mathcal{M}$ be a realizable matroid over $R$, then
		\[ \Gr \mathcal{M} = \{ (A,l) : A \in \Delta \mathcal{M}, \, l \in \tor (A)^{\vee} \}, \]
		is the \emph{set of torsions} of $\mathcal{M}$. We define an order on $\Gr \mathcal{M}$ by providing the covering relations. If $(A \cup \{b\},h),(A,l) \in \Gr \mathcal{M}$, then we set
		\[ (A,l) \triangleleft (A \cup \{b\},h) \stackrel{\text{def}}{\Longleftrightarrow} \pi_{A,b}^{\vee}(h)=l. \]
		
	\end{definition}
	
	The following example shows that the set of torsions $\Gr \mathcal{M}$ may be infinite.
	
	\begin{example}
		Set $R = \mathbb{Q}[x]$ and let $\mathcal{M}$ be the $\mathbb{Q}[x]$-matroid on the ground set $\{1\}$ defined by
		\[ \mathcal{M}(\emptyset) = \mathbb{Q}[x], \quad \mathcal{M}(\{1\}) = \nicefrac{\mathbb{Q}[x]}{(x)} \simeq \mathbb{Q}. \]
		Then the set of torsions $\Gr \mathcal{M} = \{ (\emptyset,e),(\{1\},q) : q \in \mathbb{Q} \}$ is infinite and as a poset every element $(\{1\},q)$ covers $(\emptyset,e)$, while $(\{1\},q)$ and $(\{1\},q')$ are uncomparable if $q\neq q'$.
	\end{example}

	\begin{proposition}\label{unique cover}
		Let $\mathcal{M}$ be a realizable matroid over $R$ and let $\psi:[n] \rightarrow \mathcal{M}(\emptyset)$ be a realization of $\mathcal{M}$. For every $(A,h) \in \Gr \mathcal{M}$ and $B \subseteq A$ there exists a unique $l \in \tor (B)^{\vee}$ such that $(B,l) \leq (A,h)$ in $\Gr \mathcal{M}$.
	\end{proposition}
	\begin{proof}
		We start by showing the existence. Set $A \setminus B = \{ b_1,\dots,b_m \}$, $B_0 = B$ and $B_{i+1} = B_i \cup \{ b_{i+1} \}$ for every $i \in \{ 1,\dots,m\}$. Further, set $l_0 = h$, $l_{i+1} = \pi_{B_i,b_{i+1}}^{\vee}(l_i)$ for every $i \in \{ 1,\dots,m\}$ and $l = l_m$. Then, by definition we have a chain
		\[ (B,l) \triangleleft (B_1,l_{m-1}) \triangleleft \dots \triangleleft (B_{m-1},l_1) \triangleleft (A,h) \]
		therefore $(B,l) \leq (A,h)$.
		
		About the uniqueness, it is enough to note that the choice of $l$ does not depend on the order of the elements $b_i$. In fact, for any order of the elements $b_i$, the composition of the corresponding maps $\overline{\pi_{B_i,b_{i+1}}}$ will give always the same quotient map
		\[ \pi:\mathcal{M}(B) \rightarrow \mathcal{M}(A) \simeq \mathcal{M}(B) / (\psi(j): j \in A \setminus B), \]
		and $l = \pi_{|\tor(B)}^{\vee}(h)$.
	\end{proof}
	
	We denote with $e \in \tor(\emptyset)^{\vee}$ the identity element of the torsion module $\tor(\emptyset)^{\vee}$.
	
	\begin{theorem}\label{poset of torsions simplicial poset}
		For every realizable matroid $\mathcal{M}$ over $R$ such that $\mathcal{M}(\emptyset)$ is torsion-free, $\Gr \mathcal{M}$ is a simplicial poset.
	\end{theorem}
	\begin{proof}
		The element $(\emptyset,e)$ is the minimum of $\Gr \mathcal{M}$. Now let $(A,h) \in \Gr\mathcal{M}$, and set $I = [(\emptyset,e),(A, h)]$.  From Proposition \ref{unique cover}, for every subset $B \subseteq A$, there exists a unique $l \in \tor(B)^{\vee}$ such that $(B,l) \in I$. Thus, $I$ is isomorphic, as a poset, to the Boolean lattice $2^A$.
	\end{proof}
	
	If $P$ is a poset, for every $a \in P$ we define the \emph{link} of $a$ by
	\[ \link_P a = \{ b \in P : a \leq b \}, \]
	and when there is no ambiguity about the poset, we simply write $\link a$.

	\begin{proposition}\label{link are isomorphic}
		
		Let $\mathcal{M}$ be a realizable matroid over $R$. Let $A$ be an element of $\Delta\mathcal{M}$, i.e. an independent set for the (classical) generic matroid $\mathcal{M}_E\otimes Q(R)$.
		
		\noindent
		For every $t \in \tor(A)^\vee$, $\link (A,e)$ is isomorphic to $\link (A,t)$ as poset.
		In particular, $\link (\emptyset,e)$ is isomorphic to $\link (\emptyset,t)$, for each $t \in \tor(\emptyset)^\vee$.
	\end{proposition}
	\begin{proof}
		From Proposition \ref{unique cover}, the two links do not intersect. 
		We are going to define the isomorphism from $\link (A,e)$ to $\link (A,t)$, by providing the image of the low rank element first.
		
		The atoms of $\link (A,e)$ are in bijection with the elements of the kernels of the surjective maps $\pi^{\vee}_{A,b}:\tor(A\cup \{b\})^\vee \rightarrow \tor(A)^\vee$, for all $b$ such that $A\cup \{b\}$ is an independent set of the generic matroid of $\mathcal{M}$.
		
		Similarly, the atoms of $\link (A,t)$ are in bijections with elements of the cosets $m_b + \operatorname{ker} \pi^{\vee}_{A,b}$ for the same $b$ as above and for some $m_b$ in ${\pi^{\vee}_{A,b}}^{-1}(t)$.
		The correspondence among the atoms of $\link (A,e)$ and $\link (A,t)$ is given by sending an element $x$ of $\operatorname{ker} \pi^{\vee}_{A,b}$ to $m_b+x$ in $m_b + \operatorname{ker} \pi^{\vee}_{A,b}$.

		We repeat the same construction to the atoms of $\link (A\cup \{b\},h)$, extending the map to rank two elements and so on.
		We only need to take care that the choices of $m_b$ are coherent all along the construction. This comes from the fact that $\mathcal{M}$ is realizable and, as shown in the proof of Proposition \ref{unique cover}, each map  $\pi^{\vee}_{-,-}$ is given by a specific quotient:
		\[ \begin{tikzcd}
		\tor(A)^\vee & \tor(A \cup \{b\})^\vee \arrow[l, "\pi_{A,b}^\vee"'] \\
		\tor(A\cup\{c\})^\vee \arrow[u, "\pi_{A,c}^\vee"'] & \tor(A \cup \{b, c\})^\vee \arrow[u, "\pi_{A\cup \{b\},c}^\vee"] \arrow[l, "\pi_{A\cup \{c\},b}^\vee"].
		\end{tikzcd} \]
		Therefore there exist $m_{bc}$ in  $\tor(A \cup \{b, c\})^\vee$ that maps to $m_b$ in  $\tor(A \cup \{b\})^\vee$, to $m_c$ in  $\tor(A \cup \{c\})^\vee$, and that extends the bijection among the atoms to the a bijection among the rank two elements of $\link (A,e)$ and $\link (A,t)$
		%
	\end{proof}
	
	
	\begin{proposition}\label{link simplicial poset}
		Let $\mathcal{M}$ be a realized matoid over $R$, then $\link_{\Gr \mathcal{M}}(\emptyset,e)$ is a simplicial poset.
	\end{proposition}
	\begin{proof}
		Let $\psi:[n] \rightarrow \mathcal{M}(\emptyset)$ be a realization of $\mathcal{M}$. For every $A$ subset of $[n]$ we denote by $\psi[A] = (\psi(i):i \in A)$ and one defines 
		\[ \mathcal{M}'(A) = \frac{\mathcal{M}(\emptyset)}{\tor(\emptyset) + \psi[A]} \simeq \frac{\mathcal{M}(A)}{(\tor(\emptyset) + \psi[A]) / \psi[A]}. \]
		It is clear that $\mathcal{M}'$ is a realizable matroid over $R$ and a realization is given by the composition of $\psi$ with the quotient map $\mathcal{M}(\emptyset) \rightarrow \nicefrac{\mathcal{M}(\emptyset)}{\tor(\emptyset)}$.  We want to show that $\Gr \mathcal{M}'$ is isomorphic to $\link_{\Gr \mathcal{M}} (\emptyset,e)$ as posets. Since $\mathcal{M}'(\emptyset)$ is torsion-free, from Theorem \ref{poset of torsions simplicial poset} it will follow that $\link_{\Gr \mathcal{M}}(\emptyset,e)$ is a simplicial poset. 
		Since 
		\[ 	\mathcal{M}'(A) \otimes Q(R) \simeq \frac{\mathcal{M}(\emptyset) \otimes Q(R)}{\big(\tor(\emptyset) + \psi[A]\big) \otimes Q(R)} \simeq \frac{\mathcal{M}(\emptyset) \otimes Q(R)}{\psi[A] \otimes Q(R)} \simeq \mathcal{M}(A) \otimes Q(R), \]
		then $\mathcal{M}' \otimes Q(R) \simeq \mathcal{M} \otimes Q(R)$. 
		
		For every subset $A$, let $\phi_A: \tor(A) \rightarrow \tor'(A)$ be the restriction to $\tor(A)$ of the quotient map from $\mathcal{M}(A)$ to $\mathcal{M}'(A)$, where $\tor'(A)$ denote the torsion part of $\mathcal{M}'(A)$. Consider its dual $\phi_A^{\vee}$. We want to show that
		\[ \varphi: \Gr \mathcal{M}' \rightarrow \link_{\Gr \mathcal{M}}(\emptyset,e), \quad \varphi(A,l) = (A,\phi_A^{\vee}(l)) \]
		is the desidered isomorphism. First, note that if $A \cup \{b\} \in \Delta \mathcal{M}$, the diagram
		\[ \begin{tikzcd}
		\displaystyle \frac{\mathcal{M}(\emptyset)}{\psi[A]} \arrow[r] \arrow[d] & \displaystyle\frac{\mathcal{M}(\emptyset)}{\psi[A]+\psi[b]} \arrow[d] \\
		\displaystyle\frac{\mathcal{M}(\emptyset)}{\tor(\emptyset)+\psi[A]} \arrow[r] & \displaystyle\frac{\mathcal{M}(\emptyset)}{\tor(\emptyset)+\psi[A]+\psi[b]}
		\end{tikzcd} \]
		clearly commutes, and if we restrict to the torsion parts and dualize we obtain the following commutative diagram
		\[ \begin{tikzcd}
		\tor(A)^\vee & \tor(A \cup \{b\})^\vee \arrow[l, "\pi_{A,b}^\vee"'] \\
		\tor'(A)^\vee \arrow[u, "\phi_A^\vee"'] & \tor'(A \cup \{b\})^\vee \arrow[u, "\phi_{A \cup \{b\}}^\vee"'] \arrow[l, "\pi_{A,b}'^\vee"']
		\end{tikzcd} \]
		and then
		\[ \phi_A^\vee \circ \pi_{A,b}'^\vee = \pi_{A,b}^\vee \circ \phi_{A \cup \{b\}}^\vee. \]
		The map $\phi_A$ is surjective, therefore its dual $\phi_A^\vee$ is injective. Now $(A,l) \triangleleft (A \cup \{b\},h)$ in $\Gr \mathcal{M}'$ if and only if
		\[ l = \pi_{A,b}^\vee(h) \Longleftrightarrow \phi_A^\vee(l) = \phi_A^\vee(\pi_{A,b}^\vee(h)) = \pi_{A,b}'^\vee(\phi_{A \cup \{b\}}^\vee(h)) \]
		if and only if $(A,\phi_A(l)) \triangleleft (A\cup \{b\}, \phi_{A \cup \{b\}}(h))$. 
		Hence $\varphi$ is an order-embedding, and so well defined. The injectivity of $\varphi$ follows from the injectivity of $\phi_A^\vee$. For surjectivity, first we note that whenever $\varphi(A,l') \triangleleft (A \cup \{b\},h)$, we have $(A \cup \{b\},h) \in \varphi(\Gr \mathcal{M}')$. In fact, by definition
		\[ \pi_{A,b}^\vee(h) = \phi_A^\vee(l') \Rightarrow h \in {\pi_{A,b}^\vee}^{-1}\big(\phi_A^\vee(l')\big) = \phi_{A \cup \{b\}}^\vee\big( {\pi_{A,b}'^\vee}^{-1}(l') \big). \]
		Now if $(A,l) \in \link_{\Gr \mathcal{M}}(\emptyset,e)$, then there is a chain of cover relations from $(\emptyset,e)$ to $(A,l)$, and since $(\emptyset,e) = \varphi(\emptyset,e')$, by iteratively apply the preceding remark, we obtain $(A,l) \in \varphi(\Gr \mathcal{M}')$. \qedhere

	\end{proof}
	
	The next result extend Theorem A of \cite{martino2017face} from matroids over $\mathbb{Z}$ to matroids over a domain.

	\begin{theorem}\label{thm:union-simplicial-posets}
		If $\mathcal{M}$ is a realizable matroid over $R$, then $\Gr \mathcal{M}$ is a disjoint union of simplicial posets isomorphic to $\link (\emptyset,e)$.
	\end{theorem}
	\begin{proof}
		For each $t \in \tor(\emptyset)^{\vee}$, the pair $(\emptyset,t)$ is minimal in $\Gr \mathcal{M}$, therefore, from Proposition \ref{unique cover} we have
		\[ \Gr \mathcal{M} =\bigsqcup_{t \in \tor(\emptyset)^{\vee}} \link (\emptyset,t). \]
		Finaly, from Proposition \ref{link are isomorphic}, for every $t \in \tor(\emptyset)$, $\link (\emptyset,t)$ is isomorphic to $\link (\emptyset,e)$ as posets, therefore, from Proposition \ref{link simplicial poset}, $\link(\emptyset,t)$ is a simplicial poset.
	\end{proof}
	
	\subsection{$f$-vector for the poset of torsions}\label{sec-f-vector-ref-intro}
	The poset of torsion $\Gr \mathcal{M}$ deserves an $f$-vector as any other finite simplicial posets. The $i$-th component of this vector should count the number of elements of $\Gr \mathcal{M}$ of rank $i-1$.
	Thus, the $f$-vector of the poset of torsions $\Gr \mathcal{M}$ of the matroid $\mathcal{M}$ over a domain $R$ coincides with the Grothendieck $f$-vector of the matroid $\mathcal{M}$.
	
	Whenever the poset of torsion is finite, then we get back the classical notion of the $f$-vector, by evaluating the isomorphic classes $[\tor(A)^{\vee}]$ by their cardinality, see for instance \cite{MR3336842, MR2538614}. We are going to work few examples of these cases in Section \ref{Sec:ring-of-integers}.
	
	\begin{example}\label{Example:Gr-M-with-f-vector}
		Let $R=\mathbb{Z}[i]$ and consider the matrix
		\[ (v_1,v_2) = \begin{bmatrix}
		1 & 1+i \\
		1+i & 0
		\end{bmatrix} \in R^{2,2}, \]
		where $v_1$ and $v_2$ are its columns. Let $\psi:[2] \rightarrow R^2$ with $\psi(i) = v_i$ and, for every $A \subseteq [2]$, set $\psi[A] = (\psi(i):i \in A) \subseteq R^2$. Now define $\mathcal{M}:2^{[2]} \rightarrow R$-mod with
		\[ \mathcal{M}(A) = \frac{R^2}{\psi[A]} \quad A \subseteq [2]. \]
		Thus $\mathcal{M}$ is a realizable $R$-matroid and $\psi$ is one of its realizations. More explicitely
		\[
		\begin{tikzcd}
		\mathcal{M}(\emptyset) \arrow[r] \arrow[d] & \mathcal{M}(1) \arrow[d] \\
		\mathcal{M}(2) \arrow[r] & \mathcal{M}(12)
		\end{tikzcd}
		\simeq
		\begin{tikzcd}
		\mathbb{Z}[i]^2 \arrow[r] \arrow[d] & \mathbb{Z}[i] \arrow[d] \\
		\mathbb{Z}[i] \oplus \mathbb{Z}[i]/(1+i)\mathbb{Z}[i] \arrow[r] & \mathbb{Z}[i]/2\mathbb{Z}[i]
		\end{tikzcd}
		\]
		The generic matroid of $\mathcal{M}$ is the uniform matroid $U_{2,2}$. We will see in Lemma \ref{lemma isomorphic torsion dual} that, since $R=\mathbb{Z}[i]$, in this case we have $\tor{(A)} \simeq \tor{(A)}^{\vee}$. Therefore, the poset of torsions $\Gr \mathcal{M}$ can be represented by the following diagram
		~\\
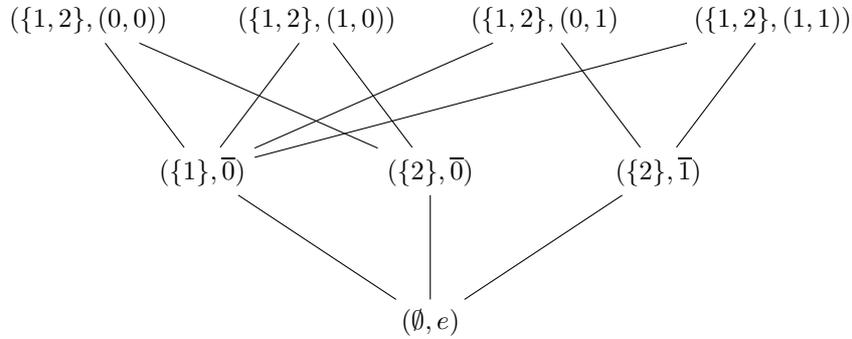
		\begin{center}
			\begin{tikzpicture}[scale=1, align=center]
			\node (00) at (0,0) {$(\emptyset,e)$};
			\node (10) at (-3,2) {$(\{1\},\overline{0})$};
			\node (20) at (0,2) {$(\{2\},\overline{0})$};
			\node (21) at (3,2) {$(\{2\},\overline{1})$};
			\node (120) at (-4.5,4) {$(\{1,2\},(0,0))$};
			\node (121) at (-1.5,4) {$(\{1,2\},(1,0))$};
			\node (122) at (1.5,4) {$(\{1,2\},(0,1)$};
			\node (123) at (4.5,4) {$(\{1,2\},(1,1))$};

			\draw (00) -- (10);
			\draw (00) -- (20);
			\draw (00) -- (21);
			
			\draw (10) -- (120);
			\draw (10) -- (121);
			\draw (10) -- (122);
			\draw (10) -- (123);
			
			\draw (20) -- (120);
			\draw (20) -- (121);
			\draw (21) -- (122);
			\draw (21) -- (123);
			
			\end{tikzpicture}
		\end{center}
		(it is the same poset of Example \ref{Example:simplicial poset}). The $f$-vector of $\Gr \mathcal{M}$ is
		\[ (f_{-1},f_0,f_1) = ([0], [0]+[\mathbb{Z}[i]/(1+i)\mathbb{Z}[i]], [\mathbb{Z}[i]/2\mathbb{Z}[i]]). \]	 
	\end{example}

	\subsection{Face module for a matroid over a domain}
	Let $\mathcal{M}$ be a realized matroid over a domain $R$, such that $\mathcal{M}(\emptyset)$ is torsion-free.
	From Theorem \ref{poset of torsions simplicial poset}, $\Gr \mathcal{M}$ is a simplicial poset and we define $A_{\mathcal{M}}$ as its face ring, see Definition \ref{def:Face-ring}.
	Whenever $\mathcal{M}(\emptyset)$ has torsions, 
	then we should define a \emph{face module}, $N_{\mathcal{M}}$.
	
	\begin{definition}\label{def:face-module}
		Let $\mathcal{M}$ be a matroid over a domain $R$, let $L$ be the link of $(\emptyset, e)$ in $\Gr \mathcal{M}$ and denote by $A_L$ the face ring of $L$.
		The face module $N_{\mathcal{M}}$ of $\mathcal{M}$ is the $A_L$-module
		\[
		N_{\mathcal{M}}=A_L^{|tor(\emptyset)|}
		\]
		(when $\tor{(\emptyset)}$ is not finite, we interpret $|\tor{(\emptyset)}|$ as a cardinal number).
	\end{definition}
	
	An alternative way to define the face module of $\mathcal{M}$ is to consider the matroid $\mathcal{M}'$ defined as in the proof of Proposition \ref{link simplicial poset}. Let $\psi:[n] \rightarrow \mathcal{M}(\emptyset)$ be a realization of $\mathcal{M}$; for every subset $A$ of $[n]$ set $\psi[A] = (\psi(i):i \in A)$ and
	\[ 
	\mathcal{M}'(A) = \frac{\mathcal{M}(\emptyset)}{\tor(\emptyset) + \psi[A]}. 
	\]
	Now $N_{\mathcal{M}} = A_\mathcal{M'}^{|\tor(\emptyset)|}$, since we have $\Gr \mathcal{M}' \simeq \link_{\Gr \mathcal{M}}(\emptyset,e)$ as posets, see proof of Proposition \ref{link simplicial poset}.

	When $R$ is a field, then $N_\mathcal{M}$ is the classical Stanley-Reiner ring of a matroid; when $R$ is the integer ring, then $N_\mathcal{M}$ is the face module for a $\mathbb{Z}$-matroid defined in \cite{martino2017face}.
	In both cases, the poset of torsion is finite and $N_\mathcal{M}$ is Noetherian. For a different ring this may not be the case.
	
	\begin{example}\label{example:simplicial poset infinite}
		If we consider the $\mathbb{Z}[x]$-matroid $\mathcal{M}$ of Example \ref{Example:Gr-M-with-f-vector}, then $\mathcal{M}(\emptyset)$ is torsion-free, so the face module of $\mathcal{M}$ is the face \emph{ring} of $P=\Gr \mathcal{M}$ which coincide with the ring $A_P$ of Example \ref{Example:simplicial poset}.
		
		Conversely, we now consider a matroid with infinite torsions. Set $R = \mathbb{Z}[x]$ and consider the matrix
		\[ (v_1,v_2) = \begin{bmatrix}
		x & 0 \\
		0 & 1
		\end{bmatrix} \in R^{2,2}, \]
		where $v_1$ and $v_2$ are its columns. Similarly as what we have done in Example \ref{Example:Gr-M-with-f-vector}, we define $\psi:[2] \rightarrow R^2$ with $\psi(i) = v_i$ and and from $\psi$ we obtain a realizable $R$-matroid $\mathcal{M}:2^{[2]} \rightarrow R$-mod
		\[ \mathcal{M}(A) = \frac{R^2}{\psi[A]} \quad A \subseteq [2]. \]
		\[
		\begin{tikzcd}
		\mathcal{M}(\emptyset) \arrow[r] \arrow[d] & \mathcal{M}(1) \arrow[d] \\
		\mathcal{M}(2) \arrow[r] & \mathcal{M}(12)
		\end{tikzcd}
		\simeq
		\begin{tikzcd}
		\mathbb{Z}[x]^2 \arrow[r] \arrow[d] & \mathbb{Z}[x] \oplus {\red \mathbb{Z}} \arrow[d] \\
		\mathbb{Z}[x] \arrow[r] & {\blue \mathbb{Z}}
		\end{tikzcd}
		\]
		(where $\mathbb{Z} \simeq \mathbb{Z}[x] / (x)$). The generic matroid of $\mathcal{M}$ is the uniform matroid $U_{2,2}$. In this case $\tor{(\mathcal{M}(1))} = \tor{(\mathcal{M}(12))} = \mathbb{Z}$ is not finite, and so is the poset of torsions $\Gr \mathcal{M}$:
		~\\
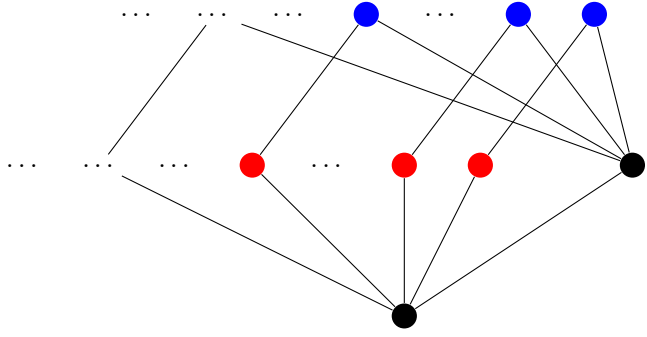
		\begin{center}
			\tikzstyle{circle1}=[shape=circle,fill=black]
			\tikzstyle{circle2}=[shape=circle,fill=red]
			\tikzstyle{circle3}=[shape=circle,fill=blue]
			\begin{tikzpicture}[scale=1, align=center]
			\node (00) at (0,0) [circle1] {};
			
			\node (dots) at (-5,2) {\dots};
			\node (dots1) at (-4,2) {\dots};
			\node (dots) at (-3,2) {\dots};
			\node (10) at (-2,2) [circle2] {};
			\node (dots) at (-1,2) {\dots};
			\node (11) at (0,2) [circle2] {};
			\node (12) at (1,2) [circle2] {};
			
			\node (13) at (3,2) [circle1] {};
			
			\node (dots) at (-3.5,4) {\dots};
			\node (dots2) at (-2.5,4) {\dots};
			\node (dots) at (-1.5,4) {\dots};
			\node (20) at (-0.5,4) [circle3] {};
			\node (dots) at (0.5,4) {\dots};
			\node (21) at (1.5,4) [circle3] {};
			\node (22) at (2.5,4) [circle3] {};

			\draw (00) -- (10);
			\draw (00) -- (11);
			\draw (00) -- (12);
			\draw (00) -- (13);
			\draw (00) -- (dots1);
			
			\draw (10) -- (20);
			\draw (11) -- (21);
			\draw (12) -- (22);
			
			\draw (13) -- (20);
			\draw (13) -- (21);
			\draw (13) -- (22);
			
			\draw (dots1) -- (dots2);
			\draw (13) -- (dots2);
			
			\end{tikzpicture}
		\end{center}
		in the preceding diagram the infinite red and blue dots corresponds respectively to the torsion part of $\mathcal{M}(1)$ and $\mathcal{M}(12)$. Note that in this case the face ring of $\Gr \mathcal{M}$ has infinite variables, therefore it is not Noetherian.
	\end{example}

	\section{Tutte polynomial for ring of integers of a number field}\label{Sec:ring-of-integers}
	
	In this section, $\mathbb{F}$ is an \emph{algebraic number field}, that is a finite field extension of the rational numbers $\mathbb{Q}$. The ring $R$ is the \emph{ring of integers} of $\mathbb{F}$, i.e. the integral closure of $\mathbb{Z}$ in $\mathbb{F}$, so $Q(R) = \mathbb{F}$. Under this hypothesis, $R$ is a Dedekind domain \cite[Theorem 9.5]{atiyah2018introduction}. Further, from \cite[Theorem 9.1.3]{alaca2004introductory}, for every nonzero ideal $I$ of $R$, the cardinality of $R/I$ is finite. We further assume that $R$ is a PID.
	
	Let $\mathcal{M}$ be a (not necessarily realizable) $R$-matroid on the ground set $[n]$ of rank $r$. From \cite[Proposition 3.3]{fink2016matroids}, for every subset $A$ of $[n]$, we can write uniquely
	\[ 
	\tor(A) = \frac{R}{I_1} \oplus \frac{R}{I_2} \oplus \dots \oplus \frac{R}{I_m} 
	\]
	for a chain $I_1 \subseteq \dots \subseteq I_m$ of nonzero ideals of $R$. 
	
	\begin{lemma}\label{lemma isomorphic torsion dual}
		For every ideal $I$ of $R$ we have
		\[ \Hom_R(\nicefrac{R}{I},\nicefrac{Q(R)}{R}) \simeq \nicefrac{R}{I}. \]
	\end{lemma}
	\begin{proof}
		Let $I = (d)$, for every $f \in \Hom_R(R/I,\nicefrac{Q(R)}{R})$ we have
		\[ d \cdot f(\overline{1}) = \overline{0} \Rightarrow f(\overline{1}) = \frac{r}{d} + R. \]
		Now it is easy to check that the maps $\varphi:\Hom_R(R/I,\nicefrac{Q(R)}{R}) \rightarrow R/I$ defined by $\varphi(f) = r+I$ is an isomorphism.
	\end{proof}
	
	\begin{corollary}\label{torsion dual same carindality}
		For every finitely generated torsion $R$-module $N$ we have $N \simeq N^{\vee}$, in particular $|N^{\vee}| = |N|$ is finite.
	\end{corollary}
	
	Now we want to define a homomorphism of rings $\varphi:L_0(R\text{-mod}) \rightarrow \mathbb{Z}$. In order to do that, since every module $M$ is isomorphic to the direct sum of its torsion part $\tor(M)$ and the free module $\nicefrac{M}{\tor(M)}$, we have to provide the values of $\varphi$ just on the isomorphic classes of projective and torsion modules. Specifically, 
	\begin{align*}
	&\varphi([F]) = 1 & \text{for every free module } F, \\
	&\varphi([N]) = |N|  & \text{for every torsion module } N.
	\end{align*}
	\noindent
	The homomorphism $\varphi$ induces the homomorphism of polynomial rings
	\[ 
	\tilde{\varphi}:L_0(R\text{-mod})[x,y] \rightarrow \mathbb{Z}[x,y]. 
	\]
	and, thus, we can consider the image under $\tilde{\varphi}$ of the Tutte polynomial
	\begin{equation}\label{eq:tilte-tutte}
	\tilde{T}_{\mathcal{M}}(x,y) = \tilde{\varphi} \big( T_{\mathcal{M}}(x,y) \big) = \sum_{A \subseteq [n]} |\tor(A)^\vee| (x-1)^{r-\rk(A)}(y-1)^{|A|-\rk(A)}
	\end{equation}
	
	\noindent
	From Corollary \ref{torsion dual same carindality} we have $|\tor(A)| = |\tor(A)^{\vee}|$ and, thus, in the case when $R=\mathbb{Z}$, then $\mathcal{M}$ has the structure of a quasi-arithmetic matroid (see \cite[Corollary 6.3]{fink2016matroids}) and $\tilde{T}_{\mathcal{M}}(x,y)$ coincides with the arithmetic Tutte polynomial defined in \ref{eq-Tutte-classical}, see \cite{moci2012tutte}.
	
	Let $\psi:[n] \rightarrow \mathcal{M}(\emptyset)$ be a realization of $\mathcal{M}$. For every $A$ subset of $[n]$ we denote by $\psi[A] = (\psi(i):i \in A)$. As in the proof of Proposition \ref{link simplicial poset}, define the $R$-matroid
	\[ \mathcal{M}'(A) = \frac{\mathcal{M}(\emptyset)}{\tor(\emptyset) + \psi[A]}. \]
	
	\begin{lemma}\label{quotient out torsion}
		Let $\mathcal{M}$ be a (even not realizable) $R$-matroid and $\mathcal{M}'$ as above, then
		\[ \tilde{T}_{\mathcal{M}}(t,1) = |\tor(\emptyset)| \tilde{T}_{\mathcal{M}'}(t,1). \]
	\end{lemma}
	\begin{proof}
		Since we are evaluating the Tutte polynomial at $y=1$, we are considering the term of the sum where $A$ is an independent set of $\Delta \mathcal{M}$. As we have seen in Section \ref{sec:poset of torsions}, the restriction of the quotient map $\tor(\emptyset) \rightarrow \tor(A)$ is injective, so by identifying $\tor(\emptyset)$ with its homomorphic image
		\[ 
		\tor(\emptyset) \simeq \frac{\tor(\emptyset)}{\tor(\emptyset) \cap \psi[A]} \simeq \frac{\tor(\emptyset) + \psi[A]}{\psi[A]}.
		\]
		We also know that
		\[ 
		\mathcal{M}'(A) \simeq \frac{\mathcal{M}(A)}{\nicefrac{(\tor(\emptyset) + \psi[A])}{\psi[A]}} \simeq \frac{\mathcal{M}(A)}{\tor(\emptyset)}. 
		\]
		This means that the torsion part of $\mathcal{M}'(A)$ is isomorphic to $\tor(A)/\tor(\emptyset)$, in particular, its cardinality is $|\tor(A)| / |\tor(\emptyset)|$ and now the statement follows easily.
	\end{proof}
	
	It is a well known results for classical matroids (see Appendix (Section A.3) of \cite{MR2492449}) and for realizable $\mathbb{Z}$-matroids \cite{martino2017face}, that the Hilbert series $N_\mathcal{M}(t)$ is a specialization of the (arithmetic) Tutte Polynomial $T_{\mathcal{M}}(\nicefrac{1}{t}, 1)$.
	We are ready to generalize such theorem also in this setting.
	
	\begin{theorem}\label{Hilbert specialization Tutte}
		If $\mathcal{M}$ is a realizable $R$-matroid of rank $r$, then
		\[ N_\mathcal{M}(t) = \frac{t^r}{(1-t)^r}\tilde{T}_{\mathcal{M}}(\nicefrac{1}{t}, 1). \]
	\end{theorem}
	\begin{proof}
		From Lemma \ref{quotient out torsion} and by the additivity property of the Hilbert series, it is enough to show that the theorem is true when $\mathcal{M}(\emptyset)$ is torsion-free, i.e. when $\tor(\emptyset) = 0$. In this case $\Gr \mathcal{M} = \link_{\Gr \mathcal{M}}(\emptyset,e)$ is a simplicial poset and $N_\mathcal{M}$ is the face ring $A_\mathcal{M}$ of $\Gr \mathcal{M}$.
		The image under $\varphi$ of the components of the $f$-vector of the matroid $\mathcal{M}$ (Definition \ref{def:f-vector}) coincides with the components $f_i$ of the $f$-vector of $\Gr \mathcal{M}$:
		\[ f_{i-1} = \sum_{\substack{ A \in \Delta \mathcal{M} \\ |A|=i }} |\tor(A)^\vee|. \]
		Thus
		\[ \begin{split}
		\tilde{T}_{\mathcal{M}}(t,1) & = \sum_{A \in \Delta \mathcal{M}}  |\tor(A)^\vee|(t-1)^{r-|A|} = \sum_{i=0}^r \left( \sum_{\substack{A \in \Delta \mathcal{M} \\ |A| = i}}  |\tor(A)^\vee| \right) (t-1)^{r-i} = \\
		& = \sum_{i=0}^r f_{i-1} (t-1)^{r-i} = \sum_{i=0}^r h_i t^{r-i}
		\end{split} \]
		(where $h_i$ are the components of the $h$-vector of $\Gr \mathcal{M}$). Finally, from Theorem \ref{hilbert series and h vector}
		\[ 
		A_\mathcal{M}(t) = \frac{1}{(1-t)^r} \sum_{i=0}^r h_it^i = \frac{t^r}{(1-t)^r}\tilde{T}_{\mathcal{M}}(1/t,1). \qedhere
		\]
	\end{proof}
	
	\begin{remark}
		The homomorphism $\varphi$, Lemma \ref{quotient out torsion} and Theorem \ref{Hilbert specialization Tutte} can be defined and proved in the more general class of rings in which every module $M$ can be decomposed in the direct sum of its torsion part $\tor{(M)}$ and $M/\tor{(M)}$, and every torsion module is finite.
	\end{remark}

	\begin{example}
		Set $R = \mathbb{Z}[i]$ (note that $\mathbb{Z}[i]$ is the ring of integers of $\mathbb{Q}[i]$ and it is a PID), and consider the $R$-matroid $\mathcal{M}$ of Example \ref{Example:Gr-M-with-f-vector}. Set $P=\Gr \mathcal{M}$, the face module of $\mathcal{M}$ is the face ring of $P$, and it was computed in Example \ref{Example:simplicial poset}:
		\[ N_\mathcal{M} = A_P \simeq \frac{\mathbb{K}[x_a,x_{b_0},x_{b_1},x_{c_0},x_{c_1},x_{d_0},x_{d_1}]}{
			\left(
			\begin{array}{l}			
			x_a x_{b_i} - (x_{c_i}+x_{d_i}), x_{b_0} x_{b_1}, \\
			x_{c_i} x_{d_j}, x_{c_0} x_{c_1}, x_{d_0} x_{d_1}, \\
			x_{b_i} x_{c_{\overline{i}}}, x_{b_i} x_{d_{\overline{i}}}
			\end{array}
			:
			\begin{array}{l}
			i,j \in \{0,1\} \\
			\overline{i} = 1-i
			\end{array}
			\right)
		} \]
		We compute the Hilbert series of $N_\mathcal{M}$ using Macaulay2 \cite{M2}:
		\[ \begin{split}
		N_\mathcal{M}(t) &= \frac{1-3t^2-2t^3+2t^4+8t^5+2t^6-12t^7-3t^8+8t^9+t^{10}-2t^{11}}{(1-t^2)^4(1-t)^3} = \\
		&= \frac{1+t+2t^2}{(1-t)^2}.
		\end{split} \]
		The image under $\varphi$ of the Tutte polynomial is
		\[ \begin{split}
		\tilde{T}_{\mathcal{M}}(x,y) &= \sum_{A \subseteq [2]} |\tor{(A)}| (x-1)^{2-\rk(A)} (y-1)^{|A|-\rk(A)} = \\
		&= (x-1)^2 + 3(x-1) + 4 = x^2 + x + 2.
		\end{split} \]
		(Note that $(1,3,4)$ and $(1,1,2)$ are the $f$-vector and the $h$-vector of $\Gr \mathcal{M}$.) Finally we have
		\[ N_{\mathcal{M}}(t) = \frac{1+t+2t^2}{(1-t)^2} = \frac{t^2}{(1-t)^2}\tilde{T}_{\mathcal{M}}(1/t,1). \]
	\end{example}
	
	\section{Elliptic Arrangements with complex multiplications}
	As an application of the previous study, in this section we treat the matroids over the ring of integers of $\mathbb{F}$, where $\mathbb{F}$ is a \emph{quadratic} field extension of $\mathbb{Q}$ for instance when the ring of Gaussian Integers $\mathbb{Z}[i]$ and the ring of Eisenstein integers $\mathbb{Z}[\rho]$.
	
	Matroids over $\mathbb{Z}$ are of great interest because of their connection to toric arrangements \cite{MR2183118}: $\mathbb{C}^*$ is the character group of $\mathbb{Z}$ and $\nicefrac{\mathbb{C}}{\mathbb{Z}}$ is analytically isomorphic to $\mathbb{C}^*$.  
	%
	Instead of $\mathbb{Z}$, we are going to consider a lattice in $\mathbb{C}$: a free, rank two, additive complex subgroup that generate the whole $\mathbb{C}$ as a real vector space.
	
	\noindent
	It is well know, that given a lattice $\Lambda$ the quotient $\nicefrac{\mathbb{C}}{\Lambda}$, equipped with the quotient topology and with an analytic structure is a Riemann surface with genus one. In particular, this is analytically isomorphic to an elliptic curve that we denote by $E(\Lambda)$; see for instance \cite{MR2931758, MR1193029}.

	\noindent
	Regular maps in between $E(\Lambda)$ that are also group homomorphims are called isogenies. One can show that $\mathbb{Z}$ is always a subring of the ring of isogenies $\operatorname{End} (E(\Lambda))$, but with respect to the lattice $\Lambda$, such ring could be larger and it could contains so called \emph{complex multiplications}.
	Specifically, if the lattice $\Lambda$ is generated by $1$ and $w$, then $E(\Lambda)$ has complex multiplications $[w]$ if and only if $w$ is quadratic over $\mathbb{Q}$.
	Moreover, in the case the curve has complex multiplication, the endomorphism ring $\operatorname{End} (E(\Lambda))$ is an order in an imaginary quadratic field.

	
	Elliptic arrangements appeared first in \cite{MR3203648}, and are also studied later by several other authors \cite{MR3487239, MR1404924, MR1274092, MR3477330, BibbyGadish-elliptic}. 
	They can be defined by a collection of regular maps $l_i$ from $E^d\rightarrow E$ and the object in study is $\mathcal{E}$, the set of \emph{hyperplanes} $\ker l_i$ in $E^d$.
	
	\noindent
	Given a lattice $\Lambda$ and assume that $\operatorname{End}(E(\Lambda))=\mathbb{Z}$, then an elliptic arrangement in $E(\Lambda)^d$ is given by a $n\times d$ integer matrix, where the $i-th$ column of the matrix defines an isogeny $l_i$. We are going to focus only on central arrangements.

	Let $T^e(x,y)$ be the elliptic Tutte polynomial associated to an central elliptic arrangement $\mathcal{E}$ defined in the end of Section 4 of \cite{MR3487239}:
	\begin{equation}\label{eq:tutte-elliptic}
	T^e_\mathcal{E}(x,y)=\sum_{S\subseteq [n]} m(S) (x-1)^{r'-\rk(S)}(y-1)^{|S|-\rk(S)}.
	\end{equation}
	Here the multiplicity $m(S)$ is the number of connected components of the intersections of elliptic hyperplanes $\cap_{i\in S} l_i$, $\rk(S)$ is the complex dimension of $\cap_{i\in S} l_i$, and $r'$ is the rank of $\mathcal{E}$.
	
	When the elliptic curve does not admit complex multiplications, then Bibby has shown that the Hilbert series of the model $A(\mathcal{E})$ for the cohomology of the open complement $U$ of the elliptic arrangement in $E^d$ is $t^nT^e_\mathcal{E}(1+\nicefrac{(1+t)^2}{t},0)$; from that she easily gets the Euler characteristic of $U$, see end of Section 4 of \cite{MR3487239}.
	
	We are able to extend her result to all elliptic  arrangements. 
	First we observe that for elliptic arrangements with complex multiplications, the elliptic Tutte polynomial $T^e_\mathcal{E}(x,y)$ is an evaluation of the Grothendieck-Tutte polynomial, see Definition \ref{Definition:GT-domain}.
	
	\begin{proposition}\label{elliptic-tutte-is-our-tutte}
		Let $\mathcal{E}$ be an elliptic arrangement in $E(\Lambda)^d$, where $E(\Lambda)$ admits complex multiplications. (Assume $\Lambda$ is generated by $1$ and $w$, where $w$ is a quadratic extension over $\mathbb{Q}$.)
		
		\noindent
		Let $R$ be $\operatorname{End} (E(\Lambda))$.
		We denote by $\mathcal{M}$ the $R$-matroid realized by $\mathcal{E}$.
		
		Then the elliptic Tutte polynomial $T^e_\mathcal{E}(x,y)$ is $\tilde{T}_{\mathcal{M}}(x,y)$, the evaluation \eqref{eq:tilte-tutte} of the Grothendieck-Tutte polynomail for the $R$-matroid $\mathcal{M}$.
	\end{proposition}
	\begin{proof}
		We observe that since $\mathcal{E}$ is a central elliptic arrangement, the matroid $\mathcal{M}$ is made by $R$-modules with finite (dual) torsion modules. Thus all results in Section \ref{Sec:ring-of-integers} still hold also for these matroids, and in particular one can evaluate, via \eqref{eq:tilte-tutte}, the Grothendieck-Tutte polynomial of $\mathcal{M}$.
		 
			
		We only need to prove, now, that $m(S)$ is $\tor S$. This fact arises from the well know result that the kernel of the isogeny $\nicefrac{\mathbb{C}}{\Lambda} \rightarrow \nicefrac{\mathbb{C}}{\Lambda}$ given by the embedding $a\Lambda\subseteq \Lambda$, where $a$ is a complex number, is $\nicefrac{\Lambda}{a\Lambda}$:
		\[
		0\rightarrow \nicefrac{\Lambda}{a\Lambda}\rightarrow \nicefrac{\mathbb{C}}{\Lambda} \rightarrow \nicefrac{\mathbb{C}}{\Lambda}\rightarrow 0.
		\]
	\end{proof}
	
	Finally we prove the extension of Bibby's Theorem:
	
	\begin{theorem}\label{thm-extension-bibby}
		Let $A(\mathcal{E})$ be the the model for the cohomology of the open complement $U$ of the elliptic arrangement in $E^d$. Then,
		\[
		A(\mathcal{E})(t)=t^nT^e_\mathcal{E}(1+\nicefrac{(1+t)^2}{t},0).
		\]
	\end{theorem}
	\begin{proof}
		If $E(\Lambda)$ does not admit complex multiplications, then the proof is provided in Section 4 of \cite{MR3487239}.
		
		If $E(\Lambda)$ has complex multiplications, then because of Proposition \ref{elliptic-tutte-is-our-tutte} we know that the elliptic Tutte is precisely the Grothendieck Tutte $T_{\mathcal{M}}(x,y)$. 
		Now, set $R=\operatorname{End} (E(\Lambda))$, where the lattice $\Lambda$ is generated by $1$ and $w$ and $R$ admit the complex multiplication $[w]$.
		We consider the realized matroid $\mathcal{M}$ over the $R$-matroid provided by $\mathcal{E}$.
		
		\noindent
		Thus, the proof follows as in the previous case, by using the non-broken circuits of the classical matroid $\mathcal{M}\otimes Q(w)$.
	\end{proof}

	\bibliographystyle{alpha}
	\bibliography{Reference}
	
	
	
	

\end{document}